\documentclass[oneside,10pt]{article}          % please do not change
     
\usepackage{amsfonts,amsmath,latexsym,amssymb} % these packages are required
\usepackage{amsthm}                % please use one of this two options for theorems
\usepackage{mathrsfs,upref}         % not so essential, but part of journal style
\usepackage{mathptmx}		    % Journal is printed with poscript fonts: 
\usepackage{oam}	            % Journal style, comment only if you find a bug
\usepackage{graphicx}
\usepackage{arydshln}
\usepackage{graphics, setspace}
\usepackage{listings}
\usepackage[b5paper,left=2cm, right=2cm, top=2.0cm, bottom=2.6cm]{geometry}	    % your paper can be easily printed on a4 or letter paper with enlargenment      
\newtheorem{theorem}{Theorem}           
\newtheorem{lemma}{Lemma}               
\newtheorem{example}{Example}

\newtheorem{conj}{Conjecture}
\newtheorem{prop}[theorem]{Proposition}
\newtheorem{defn}{Definition}

\theoremstyle{definition}

\allowdisplaybreaks

\begin{document}

\title{Examples for the Quantum Kippenhahn Theorem}

% Short title is optional, it will appear in running heads.
% It is necessary only if the title is to long to be used in running heads

\author{Ben Lawrence}

\address{\sc Ben Lawrence \\ Department of Mathematics \\ University of Auckland \\
Auckland\\
New Zealand\\
\email{blaw381@aucklanduni.ac.nz}}

%\dedicated{Dedicated to...}                    % Optional

\date{31.08.2018}                               % Please, write the date of submission

\keywords{Free analysis; non-commutative algebra; linear pencil; double eigenvalues; Kippenhahn conjecture}

\subjclass{Primary 15A22, 15A42; Secondary 46L52, 90C22}
        % AMS-2010 subj class. The list can be found on http://www.ams.org/mathscinet/msc/msc2010.html

\thanks{This work is part of the author's PhD thesis written under the supervision of Igor Klep. Supported by the Marsden Fund Council of the Royal Society of New Zealand.} 
        % Optional. Only one command thanks is allowed, use \par inside text if you need multiple thanks.       

\begin{abstract}
       Semidefinite programming optimises a linear objective function over a spectrahedron, and is one of the major advances of mathematical optimisation. Spectrahedra are described by linear pencils, which are linear matrix polynomials with hermitian matrix coefficients. Our focus will be on dimension-free linear pencils where the variables are permitted to be hermitian matrices. A major question on linear pencils, and matrix theory in general, is Kippenhahn's Conjecture. The conjecture states that given a linear pencil $xH + yK $ if the hermitian matrices $H$ and $K$ generate the full matrix algebra, then the pencil must have at least one simple eigenvalue for some $x$ and $y$. The conjecture is known to be false, via a single counterexample due to Laffey. A dimension-free version of the conjecture, known as the Quantum Kippenhahn Theorem, has recently been proven true non-constructively. We present a novel family of counterexamples to Kippenhahn's Conjecture, and use this family to construct concrete examples of the Quantum Kippenhahn Theorem.
\end{abstract}

\maketitle

%%%%% END OF TITLE PAGE %%%%%%%%%%%%%%%%%%%%%%%%%%%%%%%%%%%%%%%%%%%%%%

%%%%% BODY OF THE PAPER %%%%%%%%%%%%%%%%%%%%%%%%%%%%%%%%%%%%%%%%%%%%%%
% You should eventually delete (after reading) the rest of the text below %%

\section{Introduction}\label{intro}

Semidefinite programming is one of the major advances in mathematical optimisation of the last thirty years. It is widely used in quantitative science, in such fields as control theory, computational finance, signal processing, and fluid dynamics, among many other applications \cite{sdp, sdp_aspects}. It involves the optimisation of a linear objective function of several variables with respect to constraints represented by linear matrix inequalities on those variables. In many cases, after translation in the space of variables, a clean way of representing such constraints is with a \emph{linear pencil} and an associated \emph{linear matrix inequality}: \begin{defn}[Linear pencil and linear matrix inequality]\label{defn_pencil_LMI}
A \emph{linear pencil} is an expression of the form \begin{equation}\label{eqn_pencil_def}
L(x_{1},..,x_{n}) =  \sum_{i = 1}^{n} A_{i}x_{i}\end{equation} where the $A_{i}$ are hermitian matrices and the $x_{i} \in \mathbb{R}$ are variables. Linear pencils whose coefficients are hermitian matrices (not necessarily diagonal) give rise to linear matrix inequalities (LMIs). A \emph{linear matrix inequality} (LMI) is an expression of the form \begin{equation}\label{eqn_lmi_def}
I +  L(x_{1},..,x_{n}) \succcurlyeq 0, 
\end{equation} where $L$ is a linear pencil and the relation $\succcurlyeq$ stands for `positive semi-definite', meaning non-negative eigenvalues.
\end{defn}
Linear pencils are 
ubiquitous in matrix theory and numerical analysis
(e.g. the generalized eigenvalue problem), 
and they frequently appear in (real) algebraic geometry (cf. \cite{vinnikov,NPT}). The use of hermitian matrices (as opposed to arbitrary matrices) in the definition ensures convexity, and a full set of real eigenvalues simplify consideration of such inequalities. Inequality \eqref{eqn_lmi_def} defines a feasible region called a \emph{spectrahedron} \cite{free_convex}. Spectrahedra are always convex and semialgebraic.

The boundary of a spectrahedron associated to an linear pencil $L$  satisfies $\mbox{Det}(I + L) = 0. $ This boundary will be generically smooth, but may contain singularities such as sharp corners. The determinant of $L$ vanishes at each of the singularities. A linear function on the space of variables in which the spectrahedron lies will have a global direction of greatest increase and will have flat level sets. This means that the optimisation of a linear function over the spectrahedron will often be found at one of the singularities. To see this, consider a generically smooth convex set, with some sharp extreme points, approaching a flat plane. In many orientations, it will collide first with one of its extreme points. It is therefore of interest to be able to identify singularities algebraically. However, simply looking for vanishing points of the gradient is not always sufficient, as the following example shows.

\begin{example}[Double circle]\label{ex_double_circle}
$$ I + L =  \left( \begin{array}{cccc}
1-x & y & \, & \, \\
y & 1+x & \, & \, \\
\, & \, & 1 -y & x \\
\, & \, & x & 1+y
\end{array} \right)  \succcurlyeq 0. $$   
\end{example}This spectrahedron is a disc. Evaluating the determinant of the LMI will give us $\mbox{Det}( I + L) = (1 - x^{2} - y^{2})^{2}$, and setting this to zero gives us the boundary. However the presence of the square in the determinant means that the gradient vanishes at every point on the boundary, falsely giving the impression of singularities where there are none. However, $L$ is composed of two block diagonal submatrices, each of which determine the same feasible region. By recognizing this and selecting only one of these sub-matrices to define the LMI, we can define the same feasible region and correctly determine that there are no singularities on the boundary. In this instance, we were saved from purely algebraic singularities by $L$ being block diagonal, but will this always be true?

Our work was motivated by the 1951 conjecture of Rudolf Kippenhahn \cite{kip}: 

\begin{conj}[Kippenhahn \cite{kip}] \label{kip}
Let and $H, K$ be hermitian $2n \times 2n$ matrices, and let $f = \det(xH + yK + I) \in \mathbb{R}[x,y]$. Let $\mathcal{A}$ be the algebra generated by $H$ and $K$. If there exists $k \in \mathbb{N}, k \geq 2 $ and a $ g \in \mathbb{C}[x,y]$ such that $ f = g^{k}$ then there is some unitary matrix $U$ such that $U^{*}(xH + yK)U$ is block diagonal, and thus $ \mathcal{A} \neq M_{n}(\mathbb{C})$.
\end{conj} 
Kippenhahn's conjecture can also be illustrated geometrically. Given a linear pencil $L=I+xH+yK$ as in Conjecture \ref{kip}, its determinant $f=\det L$ gives rise to the affine scheme $\mbox{Spec} \, \mathbb{C}[x,y]/(f)$. If $f = g^{k} $ with $k \geq 2 $ then this scheme is obviously nonreduced - see for example Chapter 5, Section 3.4 of \cite{shaf}. This conjecture is now known to be false. Our objective is to extend the understanding of the counterexamples to this conjecture. Kippenhahn originally gave a more general form of this conjecture where $f$ is permitted to be a product of more than one polynomial, and the matrices need not be of even order. Kippenhahn's conjecture linked the multiplicity of eigenvalues of a certain type of matrix polynomial to the algebra generated by the coefficients of that polynomial. In his paper \cite{kip}, Kippenhahn proved that his conjecture holds for $n \leq 2$. Shapiro extended the validity range of the conjecture in a series of 1982 papers. In the first of these papers \cite{shapiro1} she demonstrated that if $f$ has a linear factor of multiplicity greater than $n/3$, then the conjecture holds. This proves the more general form of the conjecture for the case of $3 \times 3$ matrices. The second paper \cite{shapiro2} shows that if $f = g^{n/2} $ where $g$ is quadratic, then the conjecture holds. This, combined with \cite{shapiro1}, proves the conjecture for matrices of order $4$ and $5$. The final paper \cite{shapiro3} showed that the conjecture holds if $f$ is a power of a cubic factor. This is sufficient to prove the conjecture for order $6$ matrices, that is for $n = 3$ in the form we are interested in, where $f$ is a power of an irreducible polynomial, but not Kippenhahn's original more general form. Buckley recently gave a proof of the same result as a corollary to a more general result about Weierstrass cubics \cite{buckley}.

In his 1983 paper \cite{laffey}, Laffey disproved the Conjecture for $ n = 4$ with a single counterexample. In 1998 Li, Spitkovsky, and Shukla disproved Kippenhahn's more general form of the conjecture for $n = 3$ by constructing a family of counterexamples (the LSS counterexample) of the form $f = \mbox{det}(I + xH + yK) = g^{2}h $, where both $g$ and $h$ are quadratics \cite{li}.  

In \cite{burnside_graphs} we introduced a theorem which allows for the construction of a class of one-parameter families of counterexamples of Kippenhahn's conjecture, for $n \geq 4 $. Here we give a simplified version of that theorem:

\begin{theorem}[Simplified constructibility theorem \cite{burnside_graphs}]\label{thm_construct_simple}
Let $H$ and $K$ be $2n \times 2n$ hermitian matrices over $\mathbb{C}$. Then if, in a basis in which $K$ is diagonal, with any equal diagonal entries being consecutive, \begin{enumerate}
\item $K$ has eigenvalues of multiplicity at most $2$,
\item the $2 \times 2$ blocks of $H$ lying across the top row are all invertible,
\item there exist distinct  $2 \times 2$ blocks $H_{1i}$ and $H_{1j}$ lying in the top row of $H$ so that $H_{1i}H_{1i}^{*}$ and $H_{1j}H_{1j}^{*}$ do not commute,
\end{enumerate} then the algebra generated by $H$ and $K$ is $M_{2n}(\mathbb{C})$. \\
\end{theorem} 

\subsection{Main results}

By giving criteria for a pair of matrices $H$ and $K$ to generate the full matrix algebra, this theorem provides an avenue for constructing counterexamples to Kippenhahn's Conjecture. The criteria are quite broad, and allow for various families of counterexamples. We will use Theorem \ref{thm_construct_simple} in Section \ref{sec_counter} to introduce a novel counterexample family with different properties from the family introduced in \cite{burnside_graphs}. Make the following definitions: $$\alpha = \left( \begin{array}{cc}
1 & 0 \\
0 & -1
\end{array} \right), \qquad \beta = \left( \begin{array}{cc}
0 & 1 \\
1 & 0
\end{array} \right), \qquad U = \left( \begin{array}{cc}
0 & -1 \\
1 & 0
\end{array} \right).$$ The key properties of these $2 \times 2$ matrices are as follows: \begin{align*}
\alpha^{2} = \beta^{2} &= -U^{2} = I_{2}, \\
\alpha U + U \alpha = &0 = \beta U + U \beta,\\
\alpha \beta - \beta \alpha &= \left( \begin{array}{cc}
0 & 2 \\
-2 & 0 
\end{array} \right). 
\end{align*} Now let $n \geq 4$ be an integer, and define $q = \frac{n^{2}}{2} - 1 $. From the above matrices, build the following $2n \times 2n$ skew symmetric matrices, where unspecified entries are zero: \begin{equation}\label{eqn_counter_A}
A_{(b)} = {\small \left( \begin{array}{ccccccc}
q U & \alpha & \alpha & \alpha  & \alpha & \hdots & \alpha\\
-\alpha & (q+1)U & b \alpha + \beta & b \alpha \\
-\alpha & -b\alpha - \beta & (q+2)U & 0  \\
-\alpha  & -b \alpha & 0 & (q+3)U   \\
-\alpha & \, & \, & \, & (q+4)U \\
\vdots & \, & \, & \,  & \, & \ddots \\
-\alpha & \, & \, & \,  & \, & \, & (q+n-1)U\\
\end{array}\right) },
\end{equation}
 \begin{equation}\label{eqn_counter_B}
 B = {\small \left( \begin{array}{cccc}
U  \\
\, & U  \\
\, & \, & \ddots  \\
\, & \, & \, & U
\end{array} \right) },
 \end{equation}   where $b \in \mathbb{R}, b \neq 0$. We then define \begin{equation}
H_{(b)} = A_{(b)}^{2} , \qquad K = A_{(b)}B + BA_{(b)},\end{equation} giving rise to the linear pencil \begin{equation}\label{ex_family_1}
L_{(b)}(x,y) = x H_{(b)} + y K,
\end{equation} where $x$ and $y$ are real and we have indicated dependence on the parameter $b$ using a subscript. Note that $B^{2} = -I$, and consider the expression $$xA_{(b)} + yB.$$ Since this is an even-order skew symmetric matrix for all values of $x$ and $y$, its eigenvalues will come in imaginary conjugate pairs. Squaring the expression we have \begin{align*}
(xA_{(b)} + yB)^{2} &= x^{2}(A_{(b)})^{2} + xy(A_{(b)}B + BA_{(b)}) + y^{2}B^{2} \\
&= x^{2}H_{(b)} + xy K - y^{2}I,
\end{align*} which will have eigenvalues which are the squares of the eigenvalues of $A_{(b)} + y B$. Since $A_{(b)} + y B$ has imaginary conjugate pairs of eigenvalues, $x^{2}H_{(b)} + xy K - y^{2}I$ must have real negative eigenvalues of even multiplicity. Since adding a multiple of the identity only shifts eigenvalues by a constant, $x^{2}H_{(b)} + xy K$ must also have eigenvalues of even multiplicity, and with a shift of variables we can say the same for $xH_{(b)} + y K$. This claim is rigorously proven as Lemma \ref{lemma_eigenpairs} in Section \ref{sec_counter}. Care needs to be taken for the case where $y = 0$.  Evaluating $K$, we have \begin{equation}\label{eqn_K}
K = {\small \left( \begin{array}{cccc}
-2qI_{2} & \\
\, & -2(q+1)I_{2} \\
\, & \, & \ddots \\
\, & \, & \, -2(q+n-1)I_{2}
\end{array} \right) }.\end{equation} This structure emerges from the properties of $\alpha$ and $\beta$ and the way in which they cancel out in the anti-commutator of $A_{(b)}$ and $B$.  Notice how the non-zero $2 \times 2$ blocks of $K$ are simply the identity times the coefficients of the diagonal blocks of $A_{(b)}$, with a factor of $-2$. This ties in with the first condition of Theorem \ref{thm_construct_simple}. In Section \ref{sec_counter}, we will show that $H_{(b)}$ and $K$ as specified satisfy the requirements of Theorem \ref{thm_construct_simple}, and therefore span the full matrix algebra and violate Kippenhahn's conjecture. We summarise these results into the following theorem:

\begin{theorem}[New Counterexample Family]\label{thm_counter}
Let $b \neq 0$ be a real number, and let $n \geq 4$ be an integer. Then let $2n \times 2n$ skew-symmetric matrices $A_{(b)}$ and $B$ be as defined in Equations \eqref{eqn_counter_A} and \eqref{eqn_counter_B}. Let $H_{(b)} = (A_{(b)})^{2}$ and $K = A_{(b)}B + BA_{(b)}$. Then the linear pencil $$L = x H_{(b)} + y K, $$ where $x, y \in \mathbb{R}$, has the following properties: \begin{enumerate}
\item All of the eigenvalues of $L_{(b)}$ have even multiplicity,
\item $H_{(b)}$ and $K$ generate the full matrix algebra $M_{2n}(\mathbb{C})$,
\end{enumerate} for all $b \neq 0$ and for all $x, y \in \mathbb{R}$. Therefore $H_{(b)}$ and $K$ violate Kippenhahn's Conjecture.
\end{theorem}

The first key property to note is that the diagonal entries of $A_{(b)}$ are all different multiples of $U$. This ensures that $K$ has eigenvalues of multiplicity no more than two. The coefficients based on the constant $q$ have a double purpose, first purpose being to establish the proper structure of $K$, and second purpose to control the eigenvalues of $H_{(b)}$ and $K$ in a specific way. The precise meaning of this statement will be expanded upon in the latter part of this introduction, where we discuss quantisation.

The second key property of this family of counterexamples is presence of non-commuting blocks. In both cases, $A_{(b)}$, from which $H_{(b)}$ is built, is filled almost entirely off the diagonal with $\alpha$, except for $\beta$ which is placed at a specific location. To get an impression of what is happening here, consider the upper-left $8 \times 8$ block of $A_{(b)}$: $$ {\small \left( \begin{array}{cccc}
q U & \alpha & \alpha & \alpha \\
-\alpha & (q+1)U & b \alpha + \beta & b \alpha \\
-\alpha & -b\alpha - \beta & (q+2)U & 0  \\
-\alpha  & -b \alpha & 0 & (q+3)U 
\end{array}\right) }.$$ When $A_{(b)}$ is squared to give $H$, observe what happens when the second row of blocks of $A_{(b)}$ is multiplied by the third column and fourth column respectively. In the first instance, the $b \alpha$ term in the top row lines up with a zero and vanishes, and in the second instance it is an $b \alpha + \beta$ term which vanishes. Simplify by using $\alpha^{2} = I_{2}$, and we are left with non-commuting blocks in the second row of $H_{(b)}$. After symmetric permutation of rows and columns, all of the conditions for Theorem \ref{thm_construct_simple} are met. Intuitively, this $ 8 \times 8$ block contains a minimal amount of structure necessary to create the conditions for Theorem \ref{thm_construct_simple} to operate. Other such minimal counterexamples can be constructed \cite{burnside_graphs}. The single $b \alpha + \beta$ block is there to inject enough non-commutativity to enable $M_{2}(\mathbb{C})$ to be generated at a certain block-position, to be then dispersed all around the matrix. A little inspection will show that trying to perform this same construction with a $6 \times 6$ block will not work; there will not be enough room.

\subsection{Quantisation}

Many physical situations can be modelled by an LMI with the real variables replaced with non-commutative variables. Most often the variables of interest are hermitian matrices, due to their mathematical richness and physical relevance. The term \emph{free analysis} is used for the study of such LMIs; it refers to `dimension-free' since the dimension of the variables as hermitian matrices is usually left unspecified. Free analysis is widely used in operator algebra, mathematical physics, and quantum information theory \cite{operator_algebras}. See \cite[Chapter~8]{free_convex} for a survey of this material. Further information about free LMIs may be found in \cite{matricial_relaxation}. The process of replacing commutative variables with non-commutative variables is referred to as \emph{quantisation} \cite{matricial_relaxation}. 

%Consider a two coupled linear systems, such as two electronic circuits linked by a feedback circuit. A new linear system emerges in which the state space is the direct sum of the state spaces of each part, and the coefficients are matrices with entries which are non-commutative polynomials of the coefficient matrices of the distinct parts. The behaviour of the coupled system can be studied independently of the dimension of the matrices describing each distinct part. \cite[Chapter~8]{free_convex}. Once dissipation or any sort of dampening is included, an inequality enters the picture. Most often the coefficents of interest are hermitian matrices, due to their mathematical richness and physical relevance. The term \emph{free analysis} is used for the study of such LMIs; it refers to `dimension-free' in the sense that the dimension of the non-commutative variables as matrices is not important. Free analysis is widely used in operator algebra, mathematical physics, and quantum information theory \cite{operator_algebras}. See Chapter 8 of \cite{free_convex} for a survey of this material. A spectrahedron is associated with a free LMI. It is the formal union of spectrahedra emerging at each dimension $n \in \mathbb{N}$ of the non-commutative variables. Further information about free LMIs may be found in \cite{matricial_relaxation}. The process of replacing commutative variables with non-commutative variables is referred to as \emph{quantisation} \cite{matricial_relaxation}. 

In Section \ref{sec_quantise} we investigate the application of quantisation to the Kippenhahn conjecture. Applied to a linear pencil such as in the Kippenhahn conjecture, a linear pencil of the form $L = xH + yK$ gives rise to $$L(X,Y) = X \otimes H + Y \otimes K, $$ where $X$ and $Y$ are hermitian matrices. In this dimension-free setting Kippenhahn's conjecture has been proven to hold true as Corollary 5.6 of \cite{kv}:

\begin{theorem}[Quantum Kippenhahn Theorem]\label{thm_quant_kipp_intro}
If $A_{1},...,A_{g}$ generate $M_{d}(\mathbb{C}) $ as an $\mathbb{C}$-algebra, then there exist $n \in \mathbb{N}$ and $X_{1},...,X_{g} \in M_{n}(\mathbb{C}) $ such that $$\mbox{\rm Dim Ker} \left( I_{nd} + \sum_{i=1}^{g} X_{i} \otimes A_{i} \right) = 1. $$
\end{theorem}

 The implication of this result is that given any counterexample $H$ and $K$ to Kippenhahn's conjecture in the commutative setting, some hermitian matrices $X$ and $Y$ exist in the free setting for which $H$ and $K$ are no longer a counterexample, that is, $X \otimes H + Y \otimes K $ does not have a square characteristic polynomial. The proof of the theorem is not constructive in the sense that it proves that $X$ and $Y$ must exist such that the Kippenhahn conjecture holds true, but does not identify such $X$ and $Y$. We have been able to improve on this situation. For all of the known counterexamples to Kippenhahn's Conjecture, including our own, we have explicitly identified $2 \times 2 $ hermitian matrices $X$ and $Y$ which successfully quantise the counterexamples and restore Kippenhahn's conjecture in a dimension-free setting.

\begin{theorem}[Counterexample Quantisation]\label{thm_quant}
Let $H_{(b)}$ and $K$ be as in Theorem \ref{thm_counter}, to any order $n$, and let $$X = \left( \begin{array}{cc}
1 & 0 \\
0 & 0
\end{array} \right), \qquad Y = \left( \begin{array}{cc}
0 & 1/q^{2} \\
 1/q^{2} & 1
\end{array} \right), $$ where $q = \frac{(n+2)(n-2)}{2} + 1$. Then $X \otimes H_{(b)} +  Y \otimes K $ has at least 6 simple eigenvalues at almost all parameter values $b$ and thus tightens the Quantum Kippenhahn Theorem with respect to a particular family of counterexamples.
\end{theorem}

Theorem \ref{thm_quant} is the main result of Section \ref{sec_quantise}, and will be proven there. We finish this paper in Sections \ref{sec_LSS} and \ref{sec_laffey} with a short discussion of the quantisation of Laffey's counterexample and Li, Spitkovsky, and Shukla's counterexample to the more general form of the conjecture. With these results, the theorems of Klep and Vol\v ci\v c in \cite{kv} are given concrete expression.

\section{A novel counterexample to Kippenhahn's Conjecture}\label{sec_counter}

In this section we will construct a novel family of counterexamples to Kippenhahn's Conjecture as described in Theorem \ref{thm_counter}.  This family will be similar in structure to that constructed in \cite{burnside_graphs}, but will have different values along the diagonal, and a slightly different off-diagonal structure. Theorem \ref{thm_construct_simple} allows a fair amount of flexibility in the structure of the counterexamples it generates, and we have chosen the following structure so as to produce nice eigenvalue behaviour. The specifics of this will become clear in Section \ref{sec_quantise}. First we will state and prove the double eigenvalue property referred to in the Introduction:

\begin{lemma}[Double eigenvalue lemma]\label{lemma_eigenpairs}
Let $A$ and $B$ be real skew-symmetric matrices of order $2n$, and let $B^{2} = -I$. Define $H = A^{2}$ and $K = AB + BA$. Then, for all $x, y \in \mathbb{R}$, $xH + yK$ has eigenvalues only of even multiplicity. 
\end{lemma}

\begin{proof}
Consider first $Ax + By$. Take some $x_0$ and $y_0$ in $\mathbb{R}$, where $x_0 \neq 0$. The eigenvalues of the real skew-symmetric matrix $Ax_0 + By_0$ will be purely imaginary and will exist in conjugate pairs. Note that a given pair may appear more than once. Denote such pairs by $\pm i \lambda_{k}$, where $\lambda_{k}$ is real and $k$ ranges from 1 to $n$.

Then the eigenvalues of $(Ax_0 + By_0)^{2}$ will be $-\lambda^{2}_{k}$, obviously coming in pairs. Since the same pair of eigenvalues of $Ax_0 + By_0$ may occur several times, we cannot say for sure that each $-\lambda^{2}_{k}$ has multiplicity 2, but we can be sure that it has even multiplicity. Let  $v_{k}$ be an eigenvector of $Ax_0 + By_0$ with eigenvalue $i \lambda_{k}$, and $w_{k}$ be an eigenvector of $Ax_0 + By_0$ with eigenvalue $-i \lambda_{k}$. Since $v_{k}$ and $w_{k}$ belong to different eigenspaces of $Ax_0 + By_0$, the subspace $\mbox{span} \lbrace v, w \rbrace$ which they generate is two-dimensional.

 Then $(Ax_0 + By_0)^{2} v_{k} = -\lambda^{2}_{k} v_{k} $ and $(Ax_0 + By_0)^{2} w_{k} = -\lambda^{2}_{k} w_{k} $. But $A^{2} = H$, $AB + BA = K$, and $B^{2} = -I_{2n}$, so we have $$ (Ax_0 + By_{0} )^{2} = Hx_{0}^{2} + Kx_{0}y_{0} -  I_{2n}y_{0}^{2}, $$ so $$(Hx_{0}^{2} + Kx_{0}y_{0}) v_{k} =(y_{0} -\lambda^{2}_{k}) v_{k}. $$ Dividing through by $x_{0} \neq 0$ we have that $$(Hx_{0} + Ky_{0}) v_{k} =\frac{1}{x_{0}}(y_{0} -\lambda^{2}_{k}) v_{k}. $$  Therefore, $v_{k}$ is an eigenvector of $Hx_{0} + Ky_{0}$  with eigenvalue $\frac{1}{x_{0}}(y_{0} -\lambda^{2}_{k})$. Repeating this process for $w_{k}$, we see that 
$w_{k}$ is also an eigenvector of $Hx_{0} + Ky_{0}$  with eigenvalue $\frac{1}{x_{0}}(y_{0} -\lambda^{2}_{k})$. Therefore $v_k$ and $w_k$ span a two-dimensional eigenspace of $Hx_{0} + Ky_{0}$.

For the case where $x_{0} = 0$, we simply have $K y_{0}$, which has paired eigenvalues due to the continuity of eigenvalues with respect to $x_{0}$ and $y_{0}$ \cite{matrix_analysis}.

 Therefore, for every $x_{0}$ and $y_{0}$ in $\mathbb{R}$ $Hx_0 + Ky_0$ has eigenvalues all of even multiplicity. 
\end{proof}

\subsection{Definition of the counterexample family}

Theorem \ref{thm_construct_simple} gives fairly wide freedom in constructing counterexamples to Kippenhahn's Conjecture. Our objective will be to construct a counterexample with well separated pairs of eigenvalues, and then show how these eigenvalue pairs are `split' by a suitable choice of quantisation. We will ensure the separation of pairs of eigenvalues by making the counterexample diagonally dominant in a specific way.

Take an integer $n \geq 4$, real number $b \neq 0$ and let $q = \frac{n^{2}}{2} - 1$. As before, set $$\alpha = \left( \begin{array}{cc}
1 & 0 \\
0 & -1
\end{array} \right) \mbox{, } \beta = \left( \begin{array}{cc}
0 & 1 \\
1 & 0
\end{array} \right) \mbox{ and }  U = \left( \begin{array}{cc}
0 & -1 \\
1 & 0
\end{array} \right).$$ Now define $2n \times 2n$ matrices $A_{(b)}$ as in Equation \eqref{eqn_counter_A} and $ B = I_{n} \otimes U$. 

The reason for the specific choice of $q = \frac{n^{2}}{2} - 1$ will become clear when we set bounds on the eigenvalues of the system in Section \ref{sec_quantise}, as will the reason for isolating $q$ into a variable of its own rather than simply expressing the diagonal entries in terms of $n$. Define symmetric matrices $$H_{(b)} = (A_{(b)})^{2}, K = A_{(b)}B + BA_{(b)}. $$ Lemma \ref{lemma_eigenpairs} allows us to conclude that $xH_{(b)} + yK$ has eigenvalues of all even multiplicities, and its characteristic polynomial is of the form $f = g^{2}$. So we have met the first condition for Kippenhahn's conjecture.

\subsection{The matrix algebra generated} 

We will show that each of the three conditions of Theorem \ref{thm_construct_simple} applies to $H_{(b)}$ and $K$. \begin{enumerate}
\item Recall $K$ from Equation \eqref{eqn_K}. Clearly the first condition for Theorem \ref{thm_construct_simple} is satisfied.
\item Let us evaluate the $2 \times 2$ blocks of the top row and the second row of $H_{(b)}$. First the top row: $$\left( \begin{array}{cc}
 -q^2-(n-1) & 0\\
 0 & -q^2-(n-1) 
\end{array} \right), \left( \begin{array}{cc}
-2 b & -2\\
0 & -2 b
\end{array} \right), \left( \begin{array}{cc}
 b & -1\\
-3 & b 
\end{array} \right), $$ $$ \left( \begin{array}{cc}
 b & -3 \\
-3 & b
\end{array} \right), \left( \begin{array}{cc}
0 & -4 \\
 -4 & 0  
\end{array} \right),...,\left( \begin{array}{cc}
0 & -n+1 \\
 -n+1 & 0  
\end{array} \right). $$ Each of these blocks is invertible for all permitted values of $b$, that is $b \neq 0$, except for $b = \pm 3, \pm \sqrt{3}$. Now consider the second row: 

$$\left( \begin{array}{cc}
 -2 b & 0\\
 -2 & -2 b
\end{array} \right), \left( \begin{array}{cc}
-2 b^2-q^2-2 q-3 & 0\\
0 & -2 b^2-q^2-2 q-3
\end{array} \right), \left( \begin{array}{cc}
 0 & -b\\
-b & -2 
\end{array} \right), $$ $$ \left( \begin{array}{cc}
 -1 & -2 b \\
-2 b & -1
\end{array} \right), \left( \begin{array}{cc}
-1 & 0 \\
  0 & -1  
\end{array} \right),...,\left( \begin{array}{cc}
-1 & 0 \\
  0 & -1   
\end{array} \right). $$ Each of these blocks is invertible for all permitted values of $b$, except for $b = \pm \frac{1}{2}$. Therefore for any allowed value of $b$, at least one of these rows has all invertible blocks. By simultaneous symmetric permutation of $H_{(b)}$ and $K$, we can place the invertible row at the top. Such simultaneous symmetric permutation will not affect the algebra generated by $H_{(b)}$ and $K$, nor will it affect Condition 1 on $K$. Therefore $H_{(b)}$ satisfies Condition 2 in general.
\item Take the third and fourth blocks of the first row of $H_{(b)}$: $$H_{13} = \left( \begin{array}{cc}
 b & -1\\
-3 & b 
\end{array} \right), H_{14} = \left( \begin{array}{cc}
 b & -3 \\
-3 & b
\end{array} \right).  $$ Then the commutator of $H_{13}H_{13}^{T}$ and $H_{14}H_{14}^{T}$ is $\left(
\begin{array}{cc}
 0 & 48 b \\
 -48 b & 0 \\
\end{array}
\right)$ which must always be non-zero. The commutator of the products of the third and fourth blocks of the second row with their own transposes likewise evaluates to $\left(
\begin{array}{cc}
 0 & -16 b \\
 16 b & 0 \\
\end{array}
\right)$, which must also be non-zero. So no matter which row we have placed at the top in step 2, Condition 3 is satisfied.
\end{enumerate}

We can therefore conclude that $H_{(b)}$ and $K$ violate the Kippenhahn Conjecture.

\section{Quantisation of the new counterexample}\label{sec_quantise}

As mentioned in the Introduction, quantisation \cite{matricial_relaxation} of a linear pencil is the process of replacing the commutative variables $x_{i}$ with non-commutative variables $X_{i}$, typically hermitian matrices, to give an expression of the form $$L(X) = \sum X_{i} \otimes A_{i}.$$ Such quantised pencils arise particularly in coupled linear systems.

Quantising a linear pencil such as in the Kippenhahn conjecture, we have $$L(X,Y) = X \otimes H + Y \otimes K, $$ where $X$ and $Y$ are hermitian matrices. In this setting Kippenhahn's conjecture has been proven to hold true in an arbitrary number of variables as Corollary 5.6 of \cite{kv}, stated as Theorem \ref{thm_quant_kipp_intro} in the Introduction. This Quantum Kippenhahn Theorem holds true for some hermitian $X$ and $Y$, but the proof does not identify such $X$ and $Y$. We have been able to improve on this situation. For the counterexample introduced in Section \ref{sec_counter}, we have explicitly identified $2 \times 2 $ hermitian matrices $X$ and $Y$ which successfully quantise the counterexample and restore Kippenhahn's conjecture in a dimension-free setting. Our approach will be to note that the entries of $L_{(b)}$ vary smoothly with $b$, and set $b$ to zero. We will then use Gershgorin's Disc Theorem and the Fundamental Symmetric Polynomials to show the desired result for $b = 0$. We will then use results from perturbation theory \cite{kato} to extend this result to almost all values of $b$.

To begin with, let us reconsider $A_{(b)}$ as stated in Equation \eqref{eqn_counter_A}. Notice how each element of $A_{(b)}$ is a polynomial function at most linear in $b$, and recall that $q$ is fixed depending on $n$. Now set $b = 0$. It will be recalled that $b = 0$ is specifically excluded from the valid range of the counterexample, but we may still consider it as a way to inspect the behaviour of the eigenvalues of $L_{(b)} = X \otimes H_{(b)} + Y \otimes K$. Recall that $K$ does not have any dependence on $b$. 

With $b$ set to zero, we have $$A_{(0)} = {\footnotesize \left( \begin{array}{ccccccc}
q U & \alpha & \alpha & \alpha  & \alpha & \hdots & \alpha\\
-\alpha & (q+1)U & \beta & 0 \\
-\alpha & - \beta & (q+2)U & 0  \\
-\alpha  & 0 & 0 & (q+3)U   \\
-\alpha & \, & \, & \, & (q+4)U \\
\vdots & \, & \, & \,  & \, & \ddots \\
-\alpha & \, & \, & \,  & \, & \, & (q+n-1)U\\
\end{array}\right) }.$$ Notice how each $ 2 \times 2$ block is either purely diagonal (the $\alpha$ blocks) or purely off-diagonal (the $U$ and $\beta$ blocks). Therefore, we can perform a symmetric permutation of the rows and columns of $A_{(0)}$, and likewise of $B$, separating $A_{(0)}$ into two $n \times n$ diagonal and off-diagonal blocks. Remember that $A_{(0)}$ is of dimension $2n \times 2n$. This permutation splits the elements of the $\alpha$ blocks across the diagonal blocks and the elements of $U$ and $\beta$ across the off-diagonal blocks. This is simply a change of basis order, and by an abuse of notation we may continue to refer to the altered matrix as $A_{(0)}$: $$ {\footnotesize \left( \begin{array}{ccccc;{6pt/4pt}ccccc}
0 & 1 & 1 & \hdots & 1 & -q & 0 & 0 & \hdots & 0 \\
-1 & \, & \, & \, & \, & 0 & -(q+1) & 1 & \, & \, \\
-1 & \, & \, & \, & \, & 0 & -1 & -(q+2) & \, & \, \\
\vdots & \, & \, & \, & \, & \vdots & \, & \, & \ddots & \, \\
-1 &\, & \, & \, & \, & 0 & \, & \, & \, & -(q+n-1)  \\ \hdashline[6pt/4pt]
 q & 0 & 0 & \hdots & 0 & 0 & -1 & -1 & \hdots & -1\\
 0 & (q+1) & 1 & \, & \, & 1 & \, & \, & \, & \,\\
 0 & -1 & (q+2) & \, & \, & 1 & \, & \, & \, & \,\\
\vdots & \, & \, & \ddots & \, & \vdots & \, & \, & \, & \,\\
 0 & \, & \, & \, & (q+n-1) & 1 &\, & \, & \, & \,

\end{array} \right). }$$ Now we can symmetrically swap the $1^{st}$ and $(n+1)^{th}$ row, and likewise the $1^{st}$ and $(n+1)^{th}$ column, to get $$ {\footnotesize \left( \begin{array}{ccccc;{6pt/4pt}ccccc}
\, & \, & \, & \, & \, & q & -1 & -1 & \hdots & -1 \\
\, & \, & \, & \, & \, & -1 & -(q+1) & 1 & \, & \, \\
\, & \, & \, & \, & \, & -1 & -1 & -(q+2) & \, & \, \\
\, & \, & \, & \, & \, & \vdots & \, & \, & \ddots & \, \\
\, &\, & \, & \, & \, & -1 & \, & \, & \, & -(q+n-1)  \\ \hdashline[6pt/4pt]
 -q & 1 & 1 & \hdots & 1 & \, & \, & \, & \, & \,\\
 1 & (q+1) & 1 & \, & \, & \, & \, & \, & \, & \,\\
 1 & -1 & (q+2) & \, & \, & \, & \, & \, & \, & \,\\
\vdots & \, & \, & \ddots & \, & \, & \, & \, & \, & \,\\
 1 & \, & \, & \, & (q+n-1) & \, &\, & \, & \, & \,

\end{array} \right). }$$

It is important to note that, with these symmetric permutations, the eigenvalues of $H_{(0)}$ and $K$ and any linear pencils thereof, as well as the algebra generated by $H_{(0)}$ and $K$, are not affected. Therefore we can freely adjust $H_{(0)}$ and $K$ to suit our purposes, so long as the permutations are symmetric and applied equally to both $H_{(0)}$ and $K$, or $A_{(0)}$ and $B$.

We can look at $A_{(0)}$ as being composed of four blocks $$A_{(0)} = \left( \begin{array}{cc}
0 & A_{1} \\
-A_{1}^{T} & 0
\end{array} \right), $$ and therefore we will have $$H_{(0)} = \left( \begin{array}{cc}
-A_{1}A_{1}^{T} & 0 \\
0 & -A_{1}^{T}A_{1}
\end{array} \right) =  \left( \begin{array}{cc}
H_{1} & 0 \\
0 & H_{2}
\end{array} \right).$$ Likewise, $$K = \left( \begin{array}{cc}
K_{1} & 0 \\
0 & K_{1}
\end{array} \right), $$ $$K_{1} = \left( \begin{array}{cccc}
-2q & 0 & 0 & 0 \\
0 & -2(q + 1) & 0 & 0 \\
0 & 0 & \ddots & 0 \\
0 & 0 & 0 & -2(q + n - 1)
\end{array} \right). $$This is the critical difference between the counterexample family introduced here as compared with the counterexample family of \cite{burnside_graphs}. This counterexample family becomes block diagonal when the parameter $b$ is set to the forbidden value of zero, whereas the previous counterexample family of \cite{burnside_graphs} did not. Notice how $A_{1}$ is almost symmetric. It differs from its transpose only in the (2,3) and (3,2) entries.  Let us see what effect this has on $H_{1}$ and $H_{2}$. Let $$A_{1} = \left( \begin{array}{cc}
A_{2} & P \\ 
P^{T} & Q
\end{array} \right), $$ where $Q$ is diagonal, $$A_{2} = \left( \begin{array}{ccc}
q & -1 & -1 \\
-1 & -(q + 1) & 1\\
-1 & -1 & -(q + 2)
\end{array} \right), $$ $$P = \left( \begin{array}{ccc}
-1 & \, & -1 \\
 0 & \hdots & 0 \\
 0 & \, & 0 \\
\end{array} \right). $$ Let us now evaluate the two diagonal blocks of $H_{(0)}$: $$H_{1} = -\left( \begin{array}{cc}
A_{2}A_{2}^{T} + PP^{T} & A_{2}P + PQ \\
P^{T}A_{2}^{T} + QP^{T} & P^{T}P + Q^{2}
\end{array} \right), $$ $$ H_{2} = -\left( \begin{array}{cc}
A_{2}^{T}A_{2} + PP^{T} & A_{2}^{T}P + PQ \\
P^{T}A_{2} + QP^{T} & P^{T}P + Q^{2}
\end{array} \right).   $$ Observe that $$A_{2}P = A_{2}^{T}P = \left( \begin{array}{ccc}
-q & \, & -q \\
 1 & \hdots & 1 \\
 1 & \, & 1 \\
\end{array} \right). $$ Therefore $H_{1}$ and $H_{2}$ differ only in the top left $3 \times 3$ block. We can evaluate them explicitly as $${\footnotesize H_{1} = \left(
\begin{array}{cccccc}
 -q^2-n+1 & 0 & -3 & -3 & \hdots & -n+1 \\
 0 & -q^2-2 q-3 & 0 & -1 & \, & -1 \\
 -3 & 0 & -q^2-4 q-6 & -1 & \, & -1 \\
 -3 & -1 & -1 & -q^2-6 q-10 & \, & -1 \\
 \vdots & \, & \, & \, & \ddots & \vdots \\
 -n+1 & -1 & -1 & -1 & \hdots & -(q+n-1)^2-1 \\
\end{array}
\right),} $$ $${\footnotesize H_{2} = \left(
\begin{array}{cccccc}
 -q^2-n+1 & -2 & -1 & -3 & \hdots & -n+1 \\
 -2 & -q^2-2 q-3 & -2 & -1 & \, & -1 \\
 -1 & -2 & -q^2-4 q-6 & -1 & \, & -1 \\
 -3 & -1 & -1 & -q^2-6 q-10 & \, & -1 \\
 \vdots & \, & \, & \, & \ddots & \vdots \\
 -n+1 & -1 & -1 & -1 & \hdots & -(q+n-1)^2-1 \\
\end{array}
\right). }$$ Notice how the $q$ terms cancel out everywhere except along the diagonal, and additionally how there is a difference of about $-2q$ between consecutive diagonal entries in both $H_{1}$ and $H_{2}$. We are going to make this observation precise, and use the diagonal entries to get the eigenvalues under control, using the following famous theorem of Gershgorin.

\begin{theorem}[Gershgorin's Disc Theorem \cite{matrix_analysis}] \label{gershgorin}
Let $M = (m_{ij})  \in M_{n}(\mathbb{C})$, and let $$ R_{i}(M) = \sum_{j \neq i }^{n} |m_{ij}|,$$ where $1 \leq i \leq n$, denote the \emph{deleted absolute row sums} of $M$. Then all the eigenvalues of $M$ are located in the union of $n$ discs $$\bigcup_{i=1}^{n} \, \lbrace z \in \mathbb{C}: |z - m_{ii}| \leq R_{i}(M) \rbrace. $$ Furthermore, if a union of $k$ of these $n$ discs forms a connected region that is disjoint from all the remaining $n-k$ discs, then there are precisely $k$ eigenvalues of $M$ in this region. If all of the Gershgorin discs of $M$ are disjoint, then $M$ has only simple eigenvalues.
\end{theorem}

Since every matrix is similar to its transpose, the theorem may also be stated in terms of the \emph{deleted absolute column sums}. For a non-hermitian matrix this set of column discs may provide a more accurate estimate. With hermitian matrices the two sets of sums are equal.

\begin{lemma}\label{H_similiar}
$H_{1}$ and $H_{2}$ have only simple eigenvalues and are similar for all $n \geq 4$.
\end{lemma}

\begin{proof}
Let us examine the spacing of the diagonal entries of $H_{1}$ and $H_{2}$. Denote these diagonal entries by $D(t)$. The first three of these are $$D(1) = -q^2-n+1, \quad D(2) =  -q^2-2 q-3, \quad D(3) =  -q^2-4 q-6, $$ and the remainder are of the form $$D(t) = -(q  + t - 1)^{2} - 1,$$ for $4 \leq t \leq n.$ Recall the definition of $q$: \begin{align*}
q &= \frac{n^{2}}{2} - 1 \\ 
&= 1 + \frac{(n-2)(n + 2)}{2}.
\end{align*} We can evaluate \begin{align*}
D(1) - D(2) &= -n + 2q + 4 = n^{2} - n + 2 > 0.
\end{align*} Therefore we have $D(1) > D(2)$. Now compare the remaining diagonal entries: \begin{align*}
D(2) - D(3) &= 2q + 3 = n^{2} + 1 > 0, \\
D(3) - D(4) &= 2q + 4 = n^{2} + 2 > 0, \\
D(t) - D(t+1) &= -(q  + t - 1)^{2} + (q  + t)^{2} \\
&= 2(q + t) - 1  \\
&= n^2 - 2  + 2t -1 = n^{2} - 3 + 2t > 0,  \qquad \mbox{where } 4 \leq t \leq n-1.
\end{align*} So the diagonal entries of $H_{1}$ and $H_{2}$ form a descending sequence $$D(1) > D(2) > \hdots > D(n), $$ and the difference between entries always increases. The smallest difference is $D(1) - D(2)$, regardless of $n$.

Now let us examine the row sums of $H_{1}$ and $H_{2}$, as in Gershgorin's theorem. It is seen that $$R_{1}(H_{1}) = R_{1}(H_{2}) = \sum_{i = 1}^{n-1} i = \frac{1}{2}n(n-1).$$ Likewise, $$R_{2}(H_{1}) = n-3, \quad R_{2}(H_{2})= n+1,$$ $$R_{3}(H_{1}) = R_{3}(H_{2})= n, \quad R_{t}(H_{1}) = R_{t}(H_{2}) = n+t-3, \quad 4\leq t \leq n.$$ Remember that the row sums are absolute values. The largest of the row sums $ R(2),\hdots , R(n) $ is the last one, $R(n) = 2n-3$. Let us compare this to the first row sum $R(1)$: \begin{equation*}
R(1) - R(n) = \frac{1}{2}n^{2} - \frac{1}{2}n - 2n +3 = \frac{1}{2}(n-4)(n-1) + 1 > 0 \quad \forall \, n \geq 4.
\end{equation*} Therefore the radius of the first Gershgorin disc of $H_{1}$ and $H_{2}$, centered on $D(1)$, is the largest of all the Gershgorin discs. Furthermore, $$D(1) - D(2) < D(t) - D(t + 1) \quad \forall \, 2 \leq t \leq n-1. $$ This can be easily verified for $n \geq 4$, the allowed range of $n$, by referring to the separations of the $\lbrace D(t) \rbrace$ shown above. We can see that \begin{equation*}
\frac{1}{2}(D(1) - D(2)) = \frac{1}{2}(n^{2} - n + 2) > \frac{1}{2}n(n-1) = R(1).
\end{equation*} Therefore the largest of the Gershgorin discs of both $H_{1}$ and $H_{2}$ has a radius less than half of the smallest distance between two consecutive diagonal entries, when treated as points on the real line. Therefore all Gershgorin discs of $H_{1}$ and $H_{2}$, considered as separate matrices, are disjoint. Theorem \ref{gershgorin} tells us that $H_{1}$ and $H_{2}$ both therefore have only simple eigenvalues. However, $H_{(b)} = \left( \begin{array}{cc}
H_{1} & 0 \\
0 & H_{2} 
\end{array} \right)$ is the square of a skew-symmetric matrix $A_{(b)}$. Since skew-symmetric matrices have purely imaginary eigenvalues in complex-conjugate pairs, $H_{(b)}$ must be negative definite with eigenvalues appearing with even multiplicity. But because $H_{(b)}$ is block diagonal, its eigenvalues are simply the union of the eigenvalues of $H_{1}$ and $H_{2}$. We have just shown that $H_{1}$ and $H_{2}$ have only simple eigenvalues (of multiplicity 1). Therefore if $H_{1}$ and $H_{2}$ had eigenvalues $\lambda$ and $\rho$ respectively, with $\lambda \neq \rho$, then both $\lambda$ and $\rho$ would be simple eigenvalues of $H_{(b)}$. Since this is impossible, we must conclude that $H_{1}$ and $H_{2}$ must have exactly the same eigenvalues, all simple. Thus $H_{1}$ and $H_{2}$ are similar.

\end{proof}

\begin{prop}
$H_{1}$ and $H_{2}$ are not trivially similar. That is, a similarity matrix $S$ relating $H_{1}$ and $H_{2}$ cannot be of the form $$S = \left( \begin{array}{cc}
\widetilde{S} & \, \\
\, & I_{n-3}
\end{array} \right), $$ where $\widetilde{S}$ is a similarity matrix relating the the first principal $3\times 3$ blocks of $H_{1}$ and $H_{2}$.
\end{prop}

\begin{proof} 
If two matrices are similar, they must have the same determinant. Compare the determinants of the upper-left $3\times 3$ blocks of $H_{1}$ and $H_{2}$: $$\mbox{Det}{\small \left( \begin{array}{ccc}
 -q^2-n + 1 & 0 & -3 \\
 0 & -q^2-2 q-3 & 0 \\
 -3 & 0 & -q^2-4 q-6 \\
\end{array} \right) }$$ $$ = -n q^4-6 n q^3-17 n q^2-24 n q-18 n-q^6-6 q^5-16 q^4-18 q^3+8 q^2+42 q+45, $$
$$\mbox{Det} {\small \left( \begin{array}{ccc}
 -q^2-n + 1 & -2 & -1 \\
 -2 & -q^2-2 q-3 & -2 \\
 -1 & -2 & -q^2-4 q-6 \\
\end{array} \right)}$$ $$ = -n q^4-6 n q^3-17 n q^2-24 n q-14 n-q^6-6 q^5-16 q^4-18 q^3+8 q^2+42 q+33. $$ Taking the difference, we have $\mbox{Det}(H_{1}) - \mbox{Det}(H_{2}) = -4 (n-3)$, which is non-zero for all $n \geq 4$. Therefore the first principal $3\times 3$ blocks of $H_{1}$ and $H_{2}$ are not similar and so no trivial similarity matrix $S = \mbox{diag}( \widetilde{S}, I_{2n-3} )$ can be constructed for $H_{1}$ and $H_{2}$.
\end{proof}

Before proceeding we introduce some notation and give a few preliminary lemmas. 

\begin{defn}[Submatrix notation]
Given a square matrix $M$ of order $n$ and an indexing set $I \subseteq \lbrace 1, 2, ..., n \rbrace $, we may denote the principal submatrix consisting of the rows and columns indexed by $I$ as $M\lbrace I \rbrace$.
\end{defn}

\begin{defn}[Anchor point]\label{defn_anchor_point}
Given a submatrix $M\lbrace I \rbrace$ we may refer to the elements of $I$ as the \emph{anchor points} of $M\lbrace I \rbrace$.
\end{defn}

\begin{example}
Let $$M = \left( \begin{array}{cccc}
1 & 2 & 3 & 4 \\
2 & 5 & 6 & 7 \\
3 & 6 & 8 & 9 \\
4 & 7 & 9 & 0
\end{array} \right), $$ then $$M\lbrace 1,3 \rbrace  = \left( \begin{array}{cc}
1 & 3 \\
3 & 8
\end{array} \right),$$ and has anchor points $1$ and $3$.
\end{example}

%\begin{defn}[Minor notation]
%Given a submatrix $M\lbrace I \rbrace $ of a square matrix $M$, the determinant of this submatrix is denoted $\Gamma(M \lbrace I \rbrace)$ and is the minor of $M$ corresponding to the indexing set $I$.
%\end{defn}

The elementary symmetric polynomials in $n$ variables are a useful and familiar tool. We define the following notation:

\begin{defn}[Elementary symmetric polynomial on eigenvalues]\label{defn_elementary_symmetric_polys}
Let $M$ be an $n \times n$ hermitian matrix, and let $m_{1},...,m_{n}$ be its eigenvalues. Let $k \leq n$ be a natural number, and define $K$ as the set of all indexing sets of natural numbers of the form $\lbrace i_{1},i_{2} , ..., i_{k} \rbrace $ where $i_{1}<i_{2} < ...< i_{k} $. Then denote $$e_{k}(M) = \sum_{I \in K} \prod_{i \in I} m_{i}. $$ This is the $k^{th}$ elementary symmetric polynomial evaluated on the eigenvalues of $M$.
\end{defn}

\begin{example}
Suppose $M$ has eigenvalues $m_{1}, m_{2}, m_{3}$. Then $$e_{2}(M) = m_{1}m_{2} + m_{1}m_{3} + m_{2}m_{3}. $$
\end{example}

\begin{lemma}\label{lemma_elementary_symmetric_polys}
Take a matrix $M$ with eigenvalues $m_{1},...,m_{n}$. Let $k \leq n$ be a natural number, and define $K$ as the set of all indexing sets of natural numbers of the form $\lbrace i_{1},i_{2} , ..., i_{k} \rbrace $ where $i_{1}<i_{2} < ...< i_{k} $. Then $$e_{k}(M) = \sum_{I \in K} \mbox{Det } M\lbrace I \rbrace . $$
\end{lemma}

\begin{proof}
See Theorem 1.2.12 in \cite{matrix_analysis}, page 42.
\end{proof}

\begin{lemma}\label{off_diag_lemma}
The difference in determinant between matrices $\left( \begin{array}{cc}
a & b \\
c & d
\end{array} \right) \mbox{ and } \left( \begin{array}{cc}
a & e \\
f & d
\end{array} \right) $ is ${ef - bc}$. Likewise, the difference in determinant between matrices $\left( \begin{array}{cc}
a & b \\
c & d
\end{array} \right) \mbox{ and } \left( \begin{array}{cc}
e & b \\
c & f
\end{array} \right) $ is $ad - ef$.
\end{lemma}

\begin{proof}
Straightforward.
\end{proof}

%\begin{lemma}\label{P_2_similar}
%Let $$ P_{1} = \left( \begin{array}{ccc}
%a & 0 & -3 \\
%0 & b & 0 \\
%-3 & 0 & c 
%\end{array} \right)$$ and $$P_{2} = \left( \begin{array}{ccc}
%a & -2 & -1 \\
%-2 & b & -2 \\
%-1 & -2 & c 
%\end{array} \right). $$ Then $e_{2}(P_{1}) = e_{2}(P_{2})$.
%\end{lemma}
%
%\begin{proof}
%Apply Lemma \ref{off_diag_lemma} to each of the $2 \times 2$ submatrices of $P_{1}$ and $P_{2}$. The differences are \begin{align*}\mbox{Det}(P_{1} \lbrace 1, 2 \rbrace ) - \mbox{Det}(P_{2} \lbrace 1, 2 \rbrace ) &= 4, \\ \mbox{Det}(P_{1} \lbrace 1, 3 \rbrace ) - \mbox{Det}(P_{2} \lbrace 1, 3 \rbrace ) &= -8, \\ \mbox{Det}(P_{1} \lbrace 2, 3 \rbrace ) - \mbox{Det}(P_{2} \lbrace 2, 3 \rbrace ) &= 4,\end{align*} the sum of which is zero.
%\end{proof}

Now, let us define $$X = \left( \begin{array}{cc}
1 & 0 \\
0 & 0 
\end{array} \right), \quad Y = \left( \begin{array}{cc}
0 & q^{-2} \\
q^{-2} & 1 
\end{array} \right). $$

We will now set $b = 0$ and examine the eigenvalues of $$L_{(0)} = X \otimes H_{(0)} + Y \otimes K  =  \left( \begin{array}{cc}
H_{(0)} & q^{-2} K \\
q^{-2} K & K
\end{array} \right).$$ We can split $L_{(0)}$ into two $2n \times 2n$ diagonal blocks using symmetric permutation: $$L(0) = \left( \begin{array}{cc;{6pt/4pt}cc}
H_{1} & q^{-2}K_{1} & 0 & 0 \\
q^{-2}K_{1} & K_{1} & 0 & 0 \\ \hdashline[6pt/4pt]
0 & 0 & H_{2} & q^{-2}K_{1} \\
0 & 0 & q^{-2}K_{1} & K_{1}
\end{array} \right) = \left( \begin{array}{cc}
L_{1} & 0 \\
0 & L_{2}
\end{array} \right). $$ Here we are denoting the diagonal blocks of $L_{(0)}$ by $L_{1}$ and $L_{2}$.

\begin{lemma}\label{L_simple}
$L_{1}$ and $L_{2}$ each have only simple eigenvalues.
\end{lemma}

\begin{proof}
We will deal with $L_{1}$ and $L_{2}$ one at a time. 

The case of $L_{1}$: We will use Gershgorin's Disc Theorem. First examine the discs due to $H_{1}$. Here are the diagonal entries again: 
\begin{align*}
D(1) &=  - q^{2} - n + 1, & D(2) &=  - q^{2} -2q - 3, \\
D(3) &=  - q^{2} - 4 q - 6, & D(t) &=  -(q + t - 1)^{2} - 1 \qquad  4\leq t \leq n.
\end{align*}
 As in the proof of Lemma \ref{H_similiar}, the diagonal entries are in descending order and the smallest gap is between $D(1)$ and $D(2)$: $$D(1) - D(2) = n^{2} - n + 2.$$ The row sums are much as in the proof of Lemma \ref{H_similiar}, but they each pick up a single term from the presence of $q^{-2}K$:\begin{align*}
R_{1}(L_{1}) &= \frac{1}{2}n(n-1) + \frac{2}{q}, & R_{2}(L_{1}) &= n-3 + \frac{2(q + 1)}{q^{2}}, \\
R_{3}(L_{1}) &= n + \frac{2(q + 2)}{q^{2}}, & R_{t}(L_{1}) &= n+t-3 + \frac{2(q + t - 1)}{q^{2}}, \qquad 4\leq t \leq n.
\end{align*} In the proof of Lemma \ref{H_similiar} we established that of the row sums $\lbrace R_{2}(H_{1}),...,R_{n}(H_{1}) \rbrace $ it is $R_{n}(H_{1})$ that is largest. This is still true, since each of the extra terms $$\lbrace 2/q, 2(q+1)/q^2,....,2(q+t-1)/q^2,... \rbrace $$ is larger than the previous one. Therefore let us compare $R_{1}(L_{1})$ and $R_{n}(L_{1})$: \begin{align*}
R_{1}(L_{1}) - R_{n}(L_{1}) &= \frac{1}{2}n(n-1) + \frac{2}{q} - 2n + 3 - \frac{2(q + n - 1)}{q^{2}} \\
&= (n-1)\left( \frac{n-4}{2} - \frac{2}{q^{2}} \right) + 1.
\end{align*} This expression increases monotonically with $n$. At $n = 4$, $q = 7$, and the expression evaluates to \begin{equation*}
R_{1}(L_{1}) - R_{n}(L_{1}) = 3\left( - \frac{2}{7^{2}} \right) + 1 = \frac{43}{49} > 0.
\end{equation*} This difference can only increase as $n$ grows larger, so we can say that for all $n \geq 4$, $R_{1}(L_{1})$ is the largest row sum. Now let us see how the Gershgorin disc associated with $R_{1}(L_{1})$ fits into the space between $D(1)$ and $D(2)$: $$
R_{1}(L_{1}) = \frac{1}{2}n(n-1) + \frac{2}{q}, \qquad
\frac{1}{2}(D(1) - D(2)) = \frac{1}{2} n(n-1) + 1.
$$ But since $q > 2$ for all $n \geq 4$, we can see that $R_{1}(L_{1}) < \frac{1}{2}(D(1) - D(2))$ for all $n \geq 4$. We have already shown that every subsequent row sum is smaller than $R_{1}(L_{1})$ and that every pair of diagonal entries is farther apart. This means that none of the Gershgorin discs can make it halfway across the space between each pair of diagonal entries, and so all of the Gershgorin discs due to $H_{1}$ and $q^{-2}K_{1}$ are disjoint.

Now consider the lower right block of $L_{1}$ equal to $K_{1}$. The diagonal entries are now: \begin{align*}
D(1+n) &=  -2q & D(2+n) &=  -2(q+1) \\
D(3+n) &=  -2(q+2) & D(t+n) &=  -2(q+t-1) \qquad  4\leq t \leq n.
\end{align*} The space between any two diagonal entries is exactly 2. The row sums come from the off-diagonal $q^{-2} K$ block: \begin{align*}
R_{n+1}(L_{1}) &=  \frac{2}{q} & R_{n+2}(L_{1}) &=  \frac{2(q+1)}{q^{2}} \\
R_{n+3}(L_{1}) &=  \frac{2(q+2)}{q^{2}} & R_{n+t}(L_{1}) &=  \frac{2(q+t-1)}{q^{2}} \qquad  4+n\leq t \leq 2n.
\end{align*} The largest of these row sums is $R_{2n}(L_{1}):$ 
\begin{align*}
R_{2n}(L_{1}) =  \frac{2(q+n-1)}{q^{2}} &< \frac{4((n-2)(n+2) + 2n)}{(n-2)^{2}(n+2)^{2}} \\
&= \frac{4n(n + 2)}{(n-2)^{2}(n+2)^{2}}  - \frac{16}{(n-2)^{2}(n+2)^{2}} \\
&< \frac{4n(n + 2)}{(n-2)^{2}(n+2)^{2}} = \frac{4n}{(n-2)^{2}(n+2)}.
\end{align*} Notice that $n < n + 2$ and $4 \leq (n-2)^{2}$ for $n \geq 4$, and therefore $\frac{4n}{(n-2)^{2}(n+2)} \leq 1$ for  $n \geq 4$. Therefore $R_{2n}(L_{1}) < 1$  for all $n \geq 4$, and so none of the Gershgorin discs centred on the diagonal entries $D(1+n),...,D(2n)$ can intersect with any of the other Gershgorin discs centred on those diagonal entries.

We have separated the Gershgorin disc into two groups, centered around $D(1),...,D(2)$ and $D(1+n),...,D(2n)$, and shown that these groups are disjoint amongst themselves. It remains to show that the Gershgorin discs centred around $D(2n)$ and $D(1)$ do not overlap, since these are  the elements of each group which approach each other the closest. The separation of these is \begin{equation*}
D(2n) - D(1) = -2(q+n-1) + q^{2} + n - 1 = q^{2} - 2q -(n +1),
\end{equation*} and the sum of their respective row sums is \begin{equation*}
R_{1}(L_{1}) + R_{2n}(L_{1})  = \frac{1}{2}n(n-1) + \frac{2}{q} + \frac{2(q+n-1)}{q^{2}} = \frac{q^{2}n(n+1) + 8q + 4n - 4}{2
q^{2}}.
\end{equation*} The difference $D(2n) - D(1) - (R_{1}(L_{1}) + R_{2n}(L_{1}))$ is easily shown to be positive for all $n \geq 4$. Thus we see that the combined radii of the Gershgorin discs centred around $D(2n)$ and $D(1)$ is not large enough to cover the distance between them, for all $n \geq 4$.

Therefore all of the Gershgorin discs of $L_{1}$ are disjoint, and so $L_{1}$ has only simple eigenvalues.

The case for $L_{2}$: This is very much the same procedure as before, except there is a difference in the first three row sums $R_{1}(L_{2}),...,R_{3}(L_{2})$ due to the difference between $H_{1}$ and $H_{2}$. This difference does not affect the outcome, and the calculations are almost exactly the same so we will omit them here.
\end{proof}

We have shown that $L_{1}$ and $L_{2}$, the two diagonal components of $L_{(0)}$, have only simple eigenvalues. The eigenvalues of $L$ are the union of the eigenvalues of $L_{1}$ and $L_{2}$. Therefore, any repeated eigenvalues of $L$ must be eigenvalues which are shared between $L_{1}$ and $L_{2}$. The following theorem shows that there must be at least some eigenvalues which are not shared between $L_{1}$ and $L_{2}$, and which are therefore simple eigenvalues of $L_{(0)}$ as a whole.

\begin{theorem}\label{L_simple_sym}
$L_{(0)}$ has at least six simple eigenvalues.
\end{theorem}

\begin{proof}
Recall that $L_{(0)} = \mbox{diag}(L_{1},L_{2})$. Denote the eigenvalues of $L_{1}$ by $$\lambda_{1}, \hdots , \lambda_{2n}, $$ and the eigenvalues of $L_{2}$ by $$\rho_{1}, \hdots , \rho_{2n}. $$ If the eigenvalues were exactly the same, then all symmetric polynomials involving the eigenvalues would also be the same. We will compare the elementary symmetric polynomials of $L_{1}$ and $L_{2}$. 

First symmetric polynomial of eigenvalues:
this is simply the trace. $L_{1}$ and $L_{2}$ share the same trace, since $H_{1}$ differs from $H_{2}$ only in off-diagonal elements, and so we can conclude that $$e_{1}(L_{1}) = \sum_{i = 1}^{2n} \lambda_{i} =  \sum_{i = 1}^{2n} \rho_{i} = e_{1}(L_{2}).$$

Second symmetric polynomial of eigenvalues:
the second symmetric polynomial of the eigenvalues of a matrix $M$ is equal to the sums of symmetric $2 \times 2$ minors of $M$: $$e_{2}(M) = \sum_{i < j} m(i,j), $$ where $m(i,j)$ is the minor formed by deleting all rows and columns except numbers $i$ and $j$. Let us therefore compare such minors of $L_{1}$ and $L_{2}$. Keep in mind that $L_{1}$ and $L_{2}$ are the same everywhere except for the $3 \times 3$ block in the upper left: $$L_{1} = {\small \left( \begin{array}{cccc}
-q^{2} - n + 1 & 0 & -3 & \hdots\\
0 & -q^{2} - 2q - 3 & 0 \\
-3 & 0 & -q^{2} - 4q - 6 \\
\vdots & & & \ddots
\end{array} \right)},$$
$$L_{2} = {\small \left( \begin{array}{cccc}
-q^{2} - n + 1 & -2 & -1 & \hdots\\
-2 & -q^{2} - 2q - 3 & -2 \\
-1 & -2 & -q^{2} - 4q - 6 \\
\vdots & & & \ddots
\end{array} \right)}.$$ Let us now examine the possible values of $l_{1}(i,j)$ and $l_{2}(i,j)$, the $2 \times 2$ minors of $L_{1}$ and $L_{2}$. Since $L_{1}$ and $L_{2}$ differ only in the first three rows and columns, $$l_{1}(i,j) = l_{2}(i,j) \qquad \mbox{ if } i,j > 3. $$ Therefore when considering the difference between $e_{2}(L_{1})$ and $e_{2}(L_{2})$ we need to consider only those minors for which $i \leq 3$. 

Consider now $i < 3, j \geq 3$. There is now one anchor point in the first three diagonal entries of $L_{1}$, and one anchor point elsewhere. Deleting all other rows and columns for the construction of the minor will invariably delete the off-diagonal entries of the first $3 \times 3$ block. Therefore $l_{1}(i,j) = l_{2}(i,j) $ for $ i \leq 3,j > 3. $

We need therefore only consider those minors which are obtained from the top-left $3 \times 3$ block. We now have {\small \begin{align*}
e_{2}(L_{1}) = \; &\mbox{Det}\left( \begin{array}{cc}
-q^{2} - n + 1 & 0\\
0 & -q^{2} - 2q -3
\end{array} \right) + \mbox{Det}\left( \begin{array}{cc}
-q^{2} - n + 1 & -3\\
-3 & -q^{2} - 4q -6
\end{array} \right) \\
&+ \mbox{Det}\left( \begin{array}{cc}
-q^{2} -4q -6 & 0\\
0 & -q^{2} - 2q -3
\end{array} \right) + (\mbox{Shared terms}),
\end{align*}
\begin{align*}
e_{2}(L_{2}) = \; &\mbox{Det}\left( \begin{array}{cc}
-q^{2} - n + 1 & -2\\
-2 & -q^{2} - 2q -3
\end{array} \right) + \mbox{Det}\left( \begin{array}{cc}
-q^{2} - n + 1 & -1\\
-1 & -q^{2} - 4q -6
\end{array} \right) \\
&+ \mbox{Det}\left( \begin{array}{cc}
-q^{2} -4q -6 & -2\\
-2 & -q^{2} - 2q -3
\end{array} \right) + (\mbox{Shared terms}),
\end{align*}}where we have grouped the terms arising from minors anchored at diagonal entries outside the first three into a `shared terms' group. Upon evaluating these three determinants, we will get terms in $q$, and three constant terms. In both cases these constant terms sum to $-9$, and so we have $$e_{2}(L_{1}) = -9 + (\mbox{Shared terms}) = e_{2}(L_{2}). $$ Therefore the second symmetric polynomial of eigenvalues is equal for both $L_{1}$ and $L_{2}$.

Third symmetric polynomial of eigenvalues:
we do not require this one for our purposes, so for brevity it will be skipped. Using similar reasoning to the case of the second symmetric polynomial, it can be shown that the third symmetric polynomial of eigenvalues is indeed the same for both $L_{1}$ and $L_{2}$.

Fourth symmetric polynomial of eigenvalues:
now we are choosing a sequence of indices $i < j < k < m$ and using these these to construct the symmetric $4 \times 4$ minors of $L_{1}$ and $L_{2}$. Many of these minors will be shared between the two matrices, and we can avoid these when considering the difference in the overall sums. The possible choices of $i,j,k,m$ fall into categories, which we will list and consider in turn. First, split $L_{1}$ and $L_{2}$ into sectors, as follows: $$L_{t} = \left( \begin{array}{ccc}
\multicolumn{1}{c|} P & \multicolumn{1}{c|} \, & \, \\ \cline{1-1}
\, & \multicolumn{1}{c|} Q & \, \\ \cline{1-2}
\, & \, & R \\
\end{array} \right). $$ Here $P$ represents the first $3 \times 3$ block of $H_{1}$ and $H_{2}$, which is the only point of difference between $L_{1}$ and $L_{2}$. To save space, use the following notation for the diagonal entries of $P$:\begin{equation*}
P_{1} = -q^{2} - n + 1, \quad  P_{2} = -q^{2} - 2q - 3, \quad P_{3} = -q^{2} - 4q - 6. 
\end{equation*}  $Q$ represents the rest of $H_{1}$ and $H_{2}$. $R$ represents the rest of $L_{1}$ and $L_{2}$, consisting of $K$ in the bottom right corner, and the two off-diagonal copies of $q^{-2} K$. We will now look through the possibilities for constructing $4 \times 4$ minors, and identify those which have the potential to be different between $L_{1}$ and $L_{2}$.

Case 1:  $i, j , k, m$ are all drawn from $Q$ and $R$. 

Since $L_{1}$ and $L_{2}$ are identical in these sectors, minors which result from the choices of $i, j , k, m$ from these two sectors alone can be disregarded as `shared terms' when comparing $e_{4}(L_{1})$ and $e_{4}(L_{2})$. The conclusion here is that at least one of the anchor points must be drawn from $P$, a fact we will refer to again.

Case 2: $i, j , k, m$ are all drawn from $P$ and $Q$.

Since $P$ and $Q$ together make up $H_{1}$ or $H_{2}$, choosing $i, j , k, m$ from these sectors gives a sum of minors that is equal to $e_{4}(H_{1})$ or $e_{4}(H_{2})$. But we already know that $H_{1}$ and $H_{2}$ have the same eigenvalues, and so $e_{4}(H_{1}) = e_{4}(H_{2})$. Therefore minors which result from the choices of $i, j , k, m$ from these two sectors alone can be disregarded as shared terms. Therefore at least one of the anchor points must be drawn from $R$.

Case 3: $i$ drawn from $P$, and $j, k, m$ are drawn from $Q$ and $R$.

In this case none of the off-diagonal entries in $P$ are drawn into the sub-matrix. Since these are the only point of difference between $L_{1}$ and $L_{2}$, the resulting minors are shared between $e_{4}(L_{1})$ and $e_{4}(L_{2})$. Therefore at least two anchor points must come from $P$.

Case 4: $i, j$ from $P$, and $k, m$ from $R$.

 Each choice of $k$ and $m$, giving anchor points $K_{k}$ and $K_{m}$ from $R$ will result in $2 \times 2$ minors drawn from the following submatrix: $${\small \left( \begin{array}{ccc|cc}
P_{1} & (0 \mbox{ or } -2) &  (-3 \mbox{ or } -1) & \,  \, \\
(0 \mbox{ or } -2) & P_{2} & (0 \mbox{ or } -2) & r_{k} & r_{m} \\
(-3 \mbox{ or } -1) & (0 \mbox{ or } -2) & P_{3} & \, & \, \\ \hline
\, & r_{k}^{T} & \, & K_{k} & \, \\
\, & r_{m}^{T} & \, & \, &  K_{m}
\end{array}  \right)}, $$ where $r_{1}$ and $r_{2}$ represent the first three rows of the columns corresponding to $k$ and $m$. There is limited freedom in the entries of $r_{k}$ and $r_{m}$. There can be at most one non-zero entry in each one, equal to $q^{-2} K_{k}$ or $q^{-2}K_{m}$ respectively, if $k$ or $m \leq 3 + n$, and any nonzero entry in $r_{m}$ must be lower in the column than a corresponding non-zero entry in $r_{k}$. This means that it is possible, depending on the choice of $k$ and $m$ for there to be a non-zero entry in $r_{k}$ and only zero entries in $r_{m}$, but not vice-versa. All of this follows because the $R$ sector is made up entirely of multiples of the matrix $K_{1}$. Examine each of the cases in turn. Firstly, the case for all zeros: $${\small \left( \begin{array}{ccc|cc}
P_{1} & (0 \mbox{ or } -2) &  (-3 \mbox{ or } -1) & 0 & 0 \\
(0 \mbox{ or } -2) & P_{2} & (0 \mbox{ or } -2) & 0 & 0 \\
(-3 \mbox{ or } -1) & (0 \mbox{ or } -2) & P_{3} & 0 & 0 \\ \hline
0 & 0 & 0 & K_{k} & \, \\
0 & 0 & 0 & \, &  K_{m}
\end{array}  \right)}. $$ Since this matrix is block diagonal, extracting the $4 \times 4$ minors as required will lead to $$K_{k}K_{m}(2 \times 2 \mbox{ sum of minors obtained from } P ). $$ But we already know from the analysis of $e_{2}$ that these minors will be the same for $L_{1}$ and $L_{2}$, and so the $4 \times 4$ minors are the same between $L_{1}$ and $L_{2}$ in this case.

Now the case for non-zero entries in both columns. Consider first the following arrangement: $${\small \left( \begin{array}{ccc|cc}
P_{1} & (0 \mbox{ or } -2) &  (-3 \mbox{ or } -1) & q^{-2}K_{k} & 0 \\
(0 \mbox{ or } -2) & P_{2} & (0 \mbox{ or } -2) & 0 & q^{-2}K_{m} \\
(-3 \mbox{ or } -1) & (0 \mbox{ or } -2) & P_{3} & 0 & 0 \\ \hline
q^{-2}K_{k} & 0 & 0 & K_{k} & \, \\
0 &  q^{-2}K_{m} & 0 & \, &  K_{m}
\end{array}  \right)}. $$ This arrangement provides three choices of $4 \times 4$ minors in keeping with our requirements, and their sum is: {\footnotesize \begin{align*}
&\left| \begin{array}{cc;{6pt/4pt}cc}
P_{1} & (0 \mbox{ or } -2) & q^{-2}K_{k} & 0 \\
(0 \mbox{ or } -2) & P_{2} & 0 & q^{-2}K_{m} \\  \hdashline[6pt/4pt]
q^{-2}K_{k} & 0 & K_{k} & \, \\
0 &  q^{-2}K_{m} & \, &  K_{m}
\end{array} \right| + \left| \begin{array}{cc;{6pt/4pt}cc}
P_{2} & (0 \mbox{ or } -2) & 0 & q^{-2}K_{m} \\ 
(0 \mbox{ or } -2) & P_{3} & 0 & 0 \\  \hdashline[6pt/4pt]
0 & 0 & K_{k} & \, \\
q^{-2}K_{m} &  0 & \, &  K_{m}
\end{array} \right| \\
+ &\left| \begin{array}{cc;{6pt/4pt}cc}
P_{1} & (-3 \mbox{ or } -1) & q^{-2}K_{k} & 0 \\
(-3 \mbox{ or } -1) & P_{3} & 0 & 0 \\  \hdashline[6pt/4pt]
q^{-2}K_{k} & 0 & K_{k} & \, \\
0 &  0 & \, &  K_{m}
\end{array} \right|.
\end{align*} } Apply the following symmetric permutation of submatrices to the second minor:
$$ \left| \begin{array}{cc;{6pt/4pt}cc}
P_{2} & (0 \mbox{ or } -2) & 0 & q^{-2}K_{m} \\ 
(0 \mbox{ or } -2) & P_{3} & 0 & 0 \\  \hdashline[6pt/4pt]
0 & 0 & K_{k} & \, \\
q^{-2}K_{m} &  0 & \, &  K_{m}
\end{array} \right| \rightarrow \left| \begin{array}{cc;{6pt/4pt}cc}
P_{2} & (0 \mbox{ or } -2) & q^{-2}K_{m} & 0 \\ 
(0 \mbox{ or } -2) & P_{3} & 0 & 0 \\  \hdashline[6pt/4pt]
q^{-2}K_{m} & 0 & K_{m} & \, \\
0 &  0 & \, &  K_{k}
\end{array} \right|. $$ Because symmetrically permuting entries of a matrix does not change its determinant, we have not changed the value of the minor with this manipulation. Now, we will have a situation where the off-diagonal blocks of each minor commute with the bottom right block of that minor, and so we can use the following identity of block matrices, given as Theorem 3 of \cite{block_matrix_dets}:
\begin{lemma}\label{lemma_block_det}
Let $M, N, E,$ and $F$ be square matrices of equal size, where $E$ and $F$ commute. Then
 $$\left| \begin{array}{cc}
M & N \\
E & F
\end{array} \right| = \left|MF - NE \right|.$$
\end{lemma}

 The product of all of the off diagonal blocks will be a diagonal matrix, and so we will end up with an expression of the form {\small \begin{align*}
&\left| \begin{array}{cc}
\mbox{terms} & K_{k}(0 \mbox{ or } -2) \\
 K_{m}(0 \mbox{ or } -2) & \mbox{terms}
\end{array} \right| + \left| \begin{array}{cc}
\mbox{terms} & K_{m}(0 \mbox{ or } -2) \\
 K_{k}(0 \mbox{ or } -2) & \mbox{terms}
\end{array} \right| \\
+ &\left| \begin{array}{cc}
\mbox{terms} & K_{k}(-3 \mbox{ or } -1) \\
 K_{m}(-3 \mbox{ or } -1) & \mbox{terms}
\end{array} \right|.
\end{align*} }The only room for a difference to emerge in this expression between $L_{1}$ and $L_{2}$ is in the off diagonal terms. For $L_{1}$ these become $$-9K_{k}K_{m}$$ and for $L_{2}$ these become $$K_{k}K_{m}(-4-4-1) = -9K_{k}K_{m},$$ and so the terms are the same for both $L_{1}$ and $L_{2}$. This is a common feature of what we are doing here. Many expressions reduce down to the same set of $2 \times 2$ minors coming from $P$, with the same useful equality property emerging from the off-diagonals. The remaining two possibilities for non-zero entries in both columns of $R$, shown here, 
$${\small \left( \begin{array}{ccc|cc}
P_{1} & (0 \mbox{ or } -2) &  (-3 \mbox{ or } -1) & q^{-2}K_{k} & 0 \\
(0 \mbox{ or } -2) & P_{2} & (0 \mbox{ or } -2) & 0 & 0 \\
(-3 \mbox{ or } -1) & (0 \mbox{ or } -2) & P_{3} & 0 & q^{-2}K_{m} \\ \hline
q^{-2}K_{k} & 0 & 0 & K_{k} & \, \\
0 &  0 & q^{-2}K_{m} & \, &  K_{m}
\end{array}  \right)}, $$ $$ {\small \left( \begin{array}{ccc|cc}
P_{1} & (0 \mbox{ or } -2) &  (-3 \mbox{ or } -1) & 0 & 0 \\
(0 \mbox{ or } -2) & P_{2} & (0 \mbox{ or } -2) & q^{-2}K_{k} & 0 \\
(-3 \mbox{ or } -1) & (0 \mbox{ or } -2) & P_{3} & 0 & q^{-2}K_{m} \\ \hline
0 & q^{-2}K_{k} & 0 & K_{k} & \, \\
0 &  0 & q^{-2}K_{m} & \, &  K_{m}
\end{array}  \right)}, $$
work out in the same way, and will not be repeated. 

We are left with the possibility of a non-zero entry in the first column of $R$ and zeros in the second. Consider first the following arrangement: $${\small \left( \begin{array}{ccc|cc}
P_{1} & (0 \mbox{ or } -2) &  (-3 \mbox{ or } -1) & q^{-2}K_{k} & 0 \\
(0 \mbox{ or } -2) & P_{2} & (0 \mbox{ or } -2) & 0 & 0 \\
(-3 \mbox{ or } -1) & (0 \mbox{ or } -2) & P_{3} & 0 & 0 \\ \hline
q^{-2}K_{k} & 0 & 0 & K_{k} & \, \\
0 &  0 & 0 & \, &  K_{m}
\end{array}  \right)}. $$ Same procedure as in the previous step: {\footnotesize \begin{align*}
&\left| \begin{array}{cc;{6pt/4pt}cc}
P_{1} & (0 \mbox{ or } -2) & q^{-2}K_{k} & 0 \\
(0 \mbox{ or } -2) & P_{2} & 0 & 0 \\  \hdashline[6pt/4pt]
q^{-2}K_{k} & 0 & K_{k} & \, \\
0 &  0 & \, &  K_{m}
\end{array} \right| + \left| \begin{array}{cc;{6pt/4pt}cc}
P_{2} & (0 \mbox{ or } -2) & 0 & 0 \\ 
(0 \mbox{ or } -2) & P_{3} & 0 & 0 \\  \hdashline[6pt/4pt]
0 & 0 & K_{k} & \, \\
0 &  0 & \, &  K_{m}
\end{array} \right| \\
+ &\left| \begin{array}{cc;{6pt/4pt}cc}
P_{1} & (-3 \mbox{ or } -1) & q^{-2}K_{k} & 0 \\
(-3 \mbox{ or } -1) & P_{3} & 0 & 0 \\  \hdashline[6pt/4pt]
q^{-2}K_{k} & 0 & K_{k} & \, \\
0 &  0 & \, &  K_{m}
\end{array} \right|,
\end{align*} }again leading to an expression of the form \begin{align*}
&\left| \begin{array}{cc}
\mbox{terms} & K_{k}(0 \mbox{ or } -2) \\
 K_{m}(0 \mbox{ or } -2) & \mbox{terms}
\end{array} \right| + \left| \begin{array}{cc}
\mbox{terms} & K_{m}(0 \mbox{ or } -2) \\
 K_{k}(0 \mbox{ or } -2) & \mbox{terms}
\end{array} \right| \\
+ &\left| \begin{array}{cc}
\mbox{terms} & K_{k}(-3 \mbox{ or } -1) \\
 K_{m}(-3 \mbox{ or } -1) & \mbox{terms}
\end{array} \right|.
\end{align*} The remaining two possibilities for one non-zero column of $R$, $${\small \left( \begin{array}{ccc|cc}
P_{1} & (0 \mbox{ or } -2) &  (-3 \mbox{ or } -1) & 0 & 0 \\
(0 \mbox{ or } -2) & P_{2} & (0 \mbox{ or } -2) & q^{-2}K_{k} & 0 \\
(-3 \mbox{ or } -1) & (0 \mbox{ or } -2) & P_{3} & 0 & 0 \\ \hline
0 & q^{-2}K_{k} & 0 & K_{k} & \, \\
0 &  0 & 0 & \, &  K_{m}
\end{array}  \right)}, $$ $$ {\small \left( \begin{array}{ccc|cc}
P_{1} & (0 \mbox{ or } -2) &  (-3 \mbox{ or } -1) & 0 & 0 \\
(0 \mbox{ or } -2) & P_{2} & (0 \mbox{ or } -2) & 0 & 0 \\
(-3 \mbox{ or } -1) & (0 \mbox{ or } -2) & P_{3} & q^{-2}K_{k} & 0 \\ \hline
0 & 0 & q^{-2}K_{k} & K_{k} & \, \\
0 &  0 & 0 & \, &  K_{m}
\end{array}  \right)} $$ can be dealt with in the same way. The same conclusion follows as for the previous arrangements, and we see that there can be no difference between the minors of $L_{1}$ and $L_{2}$ constructed in this fashion. Therefore we can conclude that either the anchor points are of the form $i,j \in P$, $k \in Q$, $m \in R$, or of the form $i,j,k \in P$, $m \in R$.

Case 5: $i,j$ from $P$, $k$ from $P$ or $Q$, and $m$ from $R$.

Since $m$ is now definitely from $R$, we must have that $m > n$. To simplify the remaining steps, we will split the indexing of anchor points into indices $1,...,n$ for $i,j,k$ in $P$ and $Q$, and separate indices $1,...,n$ for $m$ in $R$. This amounts to redefining $m$ by $m \leftarrow m - n $. We will proceed by fixing $m$ and evaluating the possible minors from all other valid choices of $i,j$ and $k$. We will go through every $m$ from $1$ to $n$ (using the new indexing system) and collect the sum of the differences. First consider the case where $m > 3$. Consider the minor formed by $i,j,k \in P$: $${\small \left| \begin{array}{ccc|c}
-q^{2} - n + 1 & (0 \mbox{ or } -2) & (-3 \mbox{ or } -1) & 0 \\
 (0 \mbox{ or } -2) & -q^{2} - 2q - 3 &  (0 \mbox{ or } -2) & 0 \\
 (-3 \mbox{ or } -1) & (0 \mbox{ or } -2) & -q^{2} -4q - 6 & 0 \\ \hline
 0 & 0 & 0 & -2(q + m -1)
\end{array} \right|}. $$ Since $m > 3$, there are only zeros in the final row and column. Remembering that $q = (1/2)(n-2)(n+2) + 1$, we can directly evaluate this term, and we find that the difference due to this term is \begin{equation*}
\Delta(1,2,3,m) = -4(n-3)(-2(q + m -1)) = 4(n-3)(n+2)(n-2) + 8m(n-3).
\end{equation*} Now suppose that we are taking $k \in Q$. We shall fix $k$, and evaluate each of $\Delta(1,2,k,m)$, $\Delta(1,3,k,m)$, and $\Delta(2,3,k,m)$, and sum over $k$. Each minor will be of the form $${\small \left| \begin{array}{cc|c|c}
2 \times 2 & \mbox{submatrices} & (-k+1 \mbox{ or } -1) & 0 \\
 \mbox{of} & P &  -1& 0 \\ \hline
(-k+1 \mbox{ or } -1) & -1 & -(q+k-1)^{2}-1 & c \\ \hline
 0 & 0 & c & -2(q + m -1)
\end{array} \right|}, $$ where the top left block represents the three $2 \times 2$ submatrices of $P$, and $c$ can either be $0$ if $m-k \neq 0$, or $-2q^{-2}(q + m -1)$ if $m-k = 0$. This non-zero entry occurs when the row (column) originating from $k$ lines up with a non-zero entry of the top right block of $L_{1}$ (or $L_{2}$). Since $m > 3$, such lining up cannot occur from any of the rows (columns) due to $i$ or $j$, and so the remaining two entries in the fourth row and column must be zero.  For clarity, let us express the differences as follows: {\small \begin{align*}
\Delta(1,2,k,m) &= \left| \begin{array}{cc|c|c}
p_{1} & 0 & 1-k & 0 \\
0 & p_{2} & -1 & 0 \\ \hline
1-k & -1 & t & c \\ \hline
0 & 0 & c & s
\end{array} \right| - \left| \begin{array}{cc|c|c}
p_{1} & -2 & 1-k & 0 \\
-2 & p_{2} & -1 & 0 \\ \hline
1-k & -1 & t & c \\ \hline
0 & 0 & c & s
\end{array} \right|, \\
\Delta(1,3,k,m) &= \left| \begin{array}{cc|c|c}
p_{1} & -3 & 1-k & 0 \\
-3 & p_{3} & -1 & 0 \\ \hline
1-k & -1 & t & c \\ \hline
0 & 0 & c & s
\end{array} \right| - \left| \begin{array}{cc|c|c}
p_{1} & -1 & 1-k & 0 \\
-1 & p_{3} & -1 & 0 \\ \hline
1-k & -1 & t & c \\ \hline
0 & 0 & c & s
\end{array} \right|, \\
\Delta(2,3,k,m) &= \left| \begin{array}{cc|c|c}
p_{2} & 0 & -1 & 0 \\
0 & p_{3} & -1 & 0 \\ \hline
-1 & -1 & t & c \\ \hline
0 & 0 & c & s
\end{array} \right| - \left| \begin{array}{cc|c|c}
p_{2} & -2 & -1 & 0 \\
-2 & p_{3} & -1 & 0 \\ \hline
-1 & -1 & t & c \\ \hline
0 & 0 & c & s
\end{array} \right|,
\end{align*}} where $s = -2(q + m -1)$, $t = -(q+k-1)^{2}-1$, and $p_{1}, p_{2}, p_{3}$ are  the diagonal entries of $P$. Consider the Laplace expansion of these minors down the rightmost column, making use of the zero entries. In this expansion of each pair of minors, only those terms which involve both the entries at the (1,2) and (2,1) positions will differ between these two expansions. This is due to the presence of zeroes in the right column and bottom row, forcing many of the Laplacian expansion terms to equal zero. Therefore by using Lemma \ref{off_diag_lemma} to assist with the the right-column Laplacian expansions we can evaluate: \pagebreak \begin{align*}
\Delta(1,2,k,m) &= -4 \left(c^2-s (k+t-1)\right) \\
\Delta(1,3,k,m) &= 4 \left(2 c^2-s (k+2 t-1)\right) \\
\Delta(2,3,k,m) &= 4 \left(-c^2+s t+s\right),
\end{align*} which sums to $4s = -8(q + m -1)$. The contribution to $\Delta(i,j,k,m)$ for $i,j \in P$, $k \in Q$, and $m \in R, m > 3$ for a fixed $m$ and fixed $k$ is $-8(q + m -1)$. Summing this over $k \in Q$, we simply multiply by $(n-3)$ to get $$\sum_{\substack{i,j = 1 \\ i \neq j}}^{3}\sum_{k > 3} \Delta(i,j,k,m)= -8(n-3)(q + m -1)  = -4(n-3)(n+2)(n-2) - 8m(n-3).$$ But this is equal to the negative of $\Delta(1,2,3,m)$ which we obtained earlier, cancelling out. Keep in mind that this sum is for a fixed $m$. So we have $$\sum_{m > 3} \sum_{k = 1}^{n}\Delta(i,j,k,m) = 0. $$
%Examine for example $\Delta(1,2,k,m)$: \begin{align*}&\left| \begin{array}{cc|c|c}
%-q^{2} - n + 1 & 0 & -k+1 & 0 \\
% 0 & -q^{2} - 2q - 3 &  -1& 0 \\ \hline
% -k+1 & -1 & -(q+k-1)^{2} & c \\ \hline
% 0 & 0 & c & -2(q + m -1)
%\end{array} \right| \\
%&-\left| \begin{array}{cc|c|c}
%-q^{2} - n + 1 & -2 & -k+1 & 0 \\
% -2 & -q^{2} - 2q - 3 &  -1& 0 \\ \hline
% -k+1 & -1 & -(q+k-1)^{2} & c \\ \hline
% 0 & 0 & c & -2(q + m -1)
%\end{array} \right|. \end{align*} In the Laplace expansion of these determinants, only those terms which involve the entries at the (1,2) and (2,1) positions will differ between these two expansions. All others will be shared and will cancel out. Therefore by using Lemma \ref{off_diag_lemma} and subtracting matching terms in the Laplacian expansions we can evaluate \begin{align*}
%\Delta(1,2,k,m) &= -2(q + m -1)(-4(q+k-1)^{2} + 4(k+1)) - 4c^{2} \\
%\Delta(1,3,k,m) &= -2(q + m -1)(8(q+k-1)^{2} - 8(k+1)) + 8 c^{2} \\
%\Delta(2,3,k,m) &= -2(q + m -1)(-4(q+k-1)^{2} + 4(k+1)) - 4c^{2},
%\end{align*} canceling out, regardless of the value of $c$. Therefore the only contribution to $\Delta(i,j,k,m)$ where $m > n + 3$ comes from $\Delta(1,2,3,m)$, and so $$\Delta(i,j,k,m) = -4(n-3)(-2(q + m -1)), $$ for $m > 3$. Now suppose that $m \leq 3$.

 Now we shall turn our attention to the remaining case, which is $m \leq 3$. Again fix $k$, and consider $m = 1$. First, consider $\Delta(1,2,3,1) $: {\small \begin{align*}
 \Delta(1,2,3,1) &=  \left| \begin{array}{ccc|c}
-q^{2} - n + 1 & 0 & -3 & -2q^{-1} \\
 0 & -q^{2} - 2q - 3 &  0 & 0 \\
 -3 & 0 & -q^{2} -4q - 6 & 0 \\ \hline
 -2q^{-1} & 0 & 0 & -2q
\end{array} \right| \\
&-  \left| \begin{array}{ccc|c}
-q^{2} - n + 1 & -2 & -1 & -2q^{-1} \\
-2 & -q^{2} - 2q - 3 &  -2 & 0 \\
 -1 & -2 & -q^{2} -4q - 6 & 0 \\ \hline
 -2q^{-1} & 0 & 0 & -2q
\end{array} \right|,
 \end{align*} }which, repeatedly using Lemma \ref{off_diag_lemma} and expanding only those terms which do not cancel out, gives us \begin{equation}\label{eq_diff_1}
 \Delta(1,2,3,1) =\frac{8 \left((n-3) q^3-2\right)}{q^2}
 \end{equation} Likewise, by direct calculation, \begin{align}\label{eq_diff_2}
 \Delta(1,2,3,2) &= \frac{8 (q+1) \left((n-3) q^4+4 q+4\right)}{q^4}, \\
 \Delta(1,2,3,3) &= \frac{8 (q+2) \left((n-3) q^4-2 q-4\right)}{q^4}.\label{eq_diff_3}
 \end{align} Summing these, we have that \begin{align}\label{eq_diff_sum}
 \sum_{m=1}^{3}\Delta(1,2,3,m) &=  8 \left(3 n (q+1)-\frac{4}{q^4}-9 q-9\right).
\end{align} Now let us examine the case for a fixed $k > 3$ and a fixed $m <3$. Start with $m = 1,$ and use Lemma \ref{off_diag_lemma} to evaluate the differences: {\small \begin{align*}
\Delta(1,2,k,1) &= \left| \begin{array}{cc|c|c}
p_{1} & 0 & 1-k & c \\
0 & p_{2} & -1 & 0 \\ \hline
1-k & -1 & t & 0 \\ \hline
c & 0 & 0 & s
\end{array} \right| - \left| \begin{array}{cc|c|c}
p_{1} & -2 & 1-k & c \\
-2 & p_{2} & -1 & 0 \\ \hline
1-k & -1 & t & 0 \\ \hline
c & 0 & 0 & s
\end{array} \right| = 4 s (-1 + k + t), \\
\Delta(1,3,k,1) &= \left| \begin{array}{cc|c|c}
p_{1} & -3 & 1-k & c \\
-3 & p_{3} & -1 & 0 \\ \hline
1-k & -1 & t & 0 \\ \hline
c & 0 & 0 & s
\end{array} \right| - \left| \begin{array}{cc|c|c}
p_{1} & -1 & 1-k & c \\
-1 & p_{3} & -1 & 0 \\ \hline
1-k & -1 & t & 0 \\ \hline
c & 0 & 0 & s
\end{array} \right| = -4 s (-1 + k + 2 t),\\
\Delta(2,3,k,1) &= \left| \begin{array}{cc|c|c}
p_{2} & 0 & -1 & 0 \\
0 & p_{3} & -1 & 0 \\ \hline
-1 & -1 & t & 0 \\ \hline
0 & 0 & 0 & s
\end{array} \right| - \left| \begin{array}{cc|c|c}
p_{2} & -2 & -1 & 0 \\
-2 & p_{3} & -1 & 0 \\ \hline
-1 & -1 & t & 0 \\ \hline
0 & 0 & 0 & s
\end{array} \right| = 4 s (1 + t), 
\end{align*} }where $s = -2q$, $c = -2/q$, $t = -(q + k - 1)^2$, and $p_{1},p_{2},p_{3}$ are the diagonal entries of $P$. The sum of these differences is \begin{equation}
\Delta(i,j,k,1) = 4s = -8q
\end{equation} for fixed $k > 3$, and hence\begin{equation}
\sum_{k>3} \Delta(i,j,k,1) = -(n-3)8q
\end{equation} Similar calculations give us 
\begin{equation}
\sum_{k>3} \Delta(i,j,k,2) = -(n-3)8(q+1),
\end{equation}
\begin{equation}
\sum_{k>3} \Delta(i,j,k,3) = -(n-3)8(q+2).
\end{equation} for $k > 3$. Now, adding each of these to Equations \eqref{eq_diff_1}, \eqref{eq_diff_2}, \eqref{eq_diff_3}, we find that \begin{align*}
\sum \Delta(i,j,k,1) &= -(n-3)8q +\frac{8 \left((n-3) q^3-2\right)}{q^2} = -\frac{16}{q^2}, \\
\sum \Delta(i,j,k,2) &= (n-3)8(q+1) + \frac{8 (q+1) \left((n-3) q^4+4 q+4\right)}{q^4} = \frac{32 \left( q + 1\right)^{2}}{q^4},\\
\sum \Delta(i,j,k,3) &= (n-3)8(q+2) +\frac{8 (q+2) \left((n-3) q^4-2 q-4\right)}{q^4}= -\frac{16 \left( q + 2 \right)^{2}}{q^4}.
\end{align*} Notice how all of the explicit $n$-terms cancel out. These are the only differences we have uncovered which are non-zero. Let us add them up to find the final value of the difference between the fourth symmetric polynomial of eigenvalues of $L_{1}$ and $L_{2}$: \begin{equation*}
\sum\Delta(i,j,k,m) = \frac{32 (q+1)^2}{q^4}-\frac{16 (q+2)^2}{q^4}-\frac{16}{q^2} = -\frac{32}{q^{4}} \neq 0
\end{equation*}

We have shown that the difference of the fourth symmetric polynomial of $L_{1}$ and $L_{2}$ is non-zero, and so therefore $L_{1}$ and $L_{2}$ must have different eigenvalues. Therefore, $L_{(0)}$ has at least some simple eigenvalues. We will now place lower bounds on how many eigenvalues must differ between $L_{1}$ and $L_{2}$, and therefore a lower bound on how many simple eigenvalues $L_{(0)}$ must have.

Suppose that the eigenvalues of $L_{1}$ and $L_{2}$ are the same except for one eigenvalue. Denote the eigenvalues of $L_{1}$ by $$\lbrace \lambda_{1}, \lambda_{2}, \hdots , \lambda_{n} \rbrace, $$ and the eigenvalues of $L_{2}$ by $$\lbrace \lambda_{1} + t, \lambda_{2}, \hdots , \lambda_{n} \rbrace, $$ where $t \neq 0$. But then, the trace of $L_{1}$ cannot equal the trace of $L_{2}$, which we know it must since $e_{1}(L_{1}) = e_{1}(L_{2})$. Therefore $L_{1}$ and $L_{2}$ must differ in at least two eigenvalues, such that the sum of eigenvalues is the same for both: $$L_{1}: \quad \lbrace \lambda_{1}, \lambda_{2}, \lambda_{3}, \hdots \lambda_{n} \rbrace,$$ $$L_{2}: \quad \lbrace \lambda_{1} + t, \lambda_{2} - t, \lambda_{3}, \hdots \lambda_{n} \rbrace,$$ where $t \neq 0$. Now let us evaluate $e_{2}(L_{1})$ and $e_{2}(L_{2})$: \begin{align*}e_{2}(L_{1}) &= \lambda_{1}\lambda_{2} + (\lambda_{1} + \lambda_{2})\sum_{i=3}^{n} \lambda_{i} + \mbox{ shared terms }, \\
e_{2}(L_{2}) &= (\lambda_{1} + t)(\lambda_{2} - t) + (\lambda_{1} + t) \sum_{i=3}^{n} \lambda_{i} + (\lambda_{2} - t) \sum_{i=3}^{n} \lambda_{i}  + \mbox{ shared terms } \\
&= (\lambda_{1} + t)(\lambda_{2} - t) + (\lambda_{1} + \lambda_{2})\sum_{i=3}^{n} \lambda_{i}  + \mbox{ shared terms }.
 \end{align*} We know that $e_{2}(L_{1}) = e_{2}(L_{2})$, and so \begin{equation*}\lambda_{1}\lambda_{2} = (\lambda_{1} + t)(\lambda_{2} - t) = \lambda_{1}\lambda_{2} + t(\lambda_{2} - \lambda_{1}) - t^{2},
 \end{equation*} leading to $$ t(\lambda_{2} - \lambda_{1} - t) = 0. $$ Since $t \neq 0$, we have that $t = \lambda_{2} - \lambda_{1}$. But for this value of $t$, we have that $\lambda_{1} + t = \lambda_{2}$, and $\lambda_{2} + t = \lambda_{1}$, which would mean that the eigenvalues of $L_{2}$ are the same as the eigenvalues of $L_{1}$, which is a contradiction. Therefore $L_{1}$ and $L_{2}$ must differ in at least three eigenvalues. There may be more eigenvalues not in common, but three is a lower bound. Therefore, since the eigenvalues of $L_{(0)}$ are the union of  the eigenvalues of $L_{1}$ and $L_{2}$, we can finally conclude that $L_{(0)}$ has at least six simple eigenvalues.
\end{proof}

Having shown the desired condition of some simple eigenvalues at a specific value of the parameter $b$ for the Kippenhahn counterexample $L_{(b)}$, we now wish to extend this condition to cover as many values of $b$ as possible. We will use results from perturbation theory found in Chapter 2 of \cite{kato} and Chapter 5 of \cite{knopp}. 

\begin{theorem}
$L_{(b)}$ has some simple eigenvalues at almost all $b$.
\end{theorem}

\begin{proof}
Consider the characteristic polynomial of $L_{(b)}$: \begin{equation}\label{eqn_kato}
\mbox{Det}(L_{(b)} - \lambda I) = 0. 
\end{equation} As stated on page 63 of \cite{kato}, this is an algebraic equation in $\lambda$ of degree $2n$, the order of $L_{(b)}$, with coefficients which are holomorphic in $b$. It follows (see page 119 of \cite{knopp}, or page 64 of \cite{kato}) that the roots of \eqref{eqn_kato} are branches of analytic functions of $b$ with only algebraic singularities. This implies that the number of distinct eigenvalues of $L_{(b)}$ is a constant independent of $b$ except at some \emph{exceptional points} where the analytic functions representing each root of \eqref{eqn_kato} meet. Since these eigenvalue functions are analytic and susceptible to a power series expansion, there can only be a finite number of such crossing points in any compact interval of $\mathbb{R}$, and only a countable number overall. At an exceptional point, the number of distinct eigenvalues can only decrease, never increase.

We already know that $L_{(0)}$ has at least six simple eigenvalues, and according to the above reasoning these eigenvalues can only collide at a countable number of values of $b$. Therefore $L_{(b)}$ has some simple eigenvalues at almost every $ b \in \mathbb{R}$, that is at all except a measure zero set of values of $b$.
\end{proof} This result establishes Theorem \ref{thm_quant} from the Introduction. We have shown that the counterexample presented in this paper can be quantised at all orders for almost all values of the parameter $b$. The tradeoff is that we have not shown that every eigenvalue is simple for almost all $b$, but only a subset of the eigenvalues. Regardless, this constitutes a quantisation of the Kippenhahn conjecture at all orders.

\section{Quantisation of Li, Spitkovsky and Shukla's counterexample}\label{sec_LSS}

In this section we will give a general description of the process of quantising the counterexample of Li, Spitkovsky, and Shukla (the LSS counterexample) to the more general form of Kippenhahn's Conjecture. First we will establish a preliminary lemma.

\begin{lemma}\label{lemma_aux}
Let $M = \left( \begin{array}{cc}
P & Q \\
Q & 0
\end{array}\right)$  be a matrix over a field $\mathbb{F}$, where $P$ and $Q$ are symmetric and invertible. Then for any eigenvalue $\lambda$ of $M$ of multiplicity $m$, the \emph{auxiliary matrix} $\widetilde{M}(\lambda)$ of $M$, defined as \begin{equation}\label{eqn_aux_matrix}
\widetilde{M}(\lambda) = P + \frac{1}{\lambda} Q^{2}
\end{equation} has $\lambda$ as an eigenvalue.

If $m > 1$, then $\lambda$ must also have multiplicity greater than 1 with respect to $\widetilde{M}(\lambda)$. 
\end{lemma}

\begin{proof}
Let $\left( \begin{array}{c}
u \\
v
\end{array} \right) $ be an eigenvector of $M$, with eigenvalue $\lambda$: $$M \left( \begin{array}{c}
u \\
v
\end{array} \right) = \left( \begin{array}{c}
P u + Q v \\
Q u
\end{array} \right)=\left( \begin{array}{c}
\lambda u \\
\lambda v
\end{array} \right).$$ Since $\mbox{Det}(M) = \mbox{Det}(Q)^{2} \neq 0 $, we know that $M$ is invertible and $\lambda \neq 0$. Comparing components, we see that $$v = \frac{1}{\lambda} Q u .$$ We then have $$\widetilde{M}(\lambda)u = P u + \frac{1}{\lambda}Q^{2} u = \lambda u.$$ So $u$ is an eigenvector of the auxiliary matrix $\widetilde{M}(\lambda)$ with eigenvalue $\lambda$. Note that $u \neq 0$ because $u$ is an eigenvector. Suppose now that $\lambda$ appeared with multiplicity greater than one as an eigenvalue of $M$. Then there must be some $\lambda$-eigenvector $\left( \begin{array}{c}
u_{2} \\
v_{2}
\end{array} \right) \neq k \left( \begin{array}{c}
u \\
v
\end{array} \right) $ for any $k \in \mathbb{F}$. By the above reasoning, both $u$ and $u_{2}$ are eigenvectors of $\widetilde{M}(\lambda)$ with eigenvalue $\lambda$. Suppose that $u_{2} = k u$. Then $v_{2} = (k/\lambda) Q u $, as above, and so we would have $\left( \begin{array}{c}
u_{2} \\
v_{2}
\end{array} \right) = k \left( \begin{array}{c}
u \\
v
\end{array} \right), $ which is not possible. Therefore we must conclude that $u_{2}$ is not a multiple of $u$. Therefore $\lambda$ must be an eigenvalue of $\widetilde{M}(\lambda)$ of multiplicity at least two, thus proving the result.
\end{proof}
The auxiliary matrix \eqref{eqn_aux_matrix} will be of use in this section, and in Section \ref{sec_laffey} where we quantise Laffey's seminal counterexample \cite{laffey}. 

In their paper \cite{li}, Li, Spitkovsky, and Shukla described the following $6 \times 6$ counterexample to the weak form of Kippenhahn's conjecture, where the characteristic polynomial need not be square. Take a matrix $A$ defined by $$A = {\small \left(
\begin{array}{cccccc}
 0 & x & 0 & c y & 0 & 0 \\
 0 & 0 & y & 0 & 0 & 0 \\
 0 & 0 & 0 & 0 & 0 & 0 \\
 0 & 0 & -c x & 0 & \sqrt{1-c^2} \xi  & 0 \\
 0 & 0 & 0 & 0 & 0 & \eta  \\
 0 & 0 & 0 & 0 & 0 & 0 \\
\end{array}
\right), } $$ where $x,$ $y,$ $\xi,$ $\eta ,$ $c > 0,$ and $x^2 + y^2 = \xi^2 + \eta^2 = 1 ,$ $c < 1/2$. Then define $H = A + A^T $ and $K$ via $2A = H + i K$: $$H = {\footnotesize \left(
\begin{array}{cccccc}
 0 & x & 0 & c y & 0 & 0 \\
 x & 0 & y & 0 & 0 & 0 \\
 0 & y & 0 & -c x & 0 & 0 \\
 c y & 0 & -c x & 0 & \sqrt{1-c^2} \xi  & 0 \\
 0 & 0 & 0 & \sqrt{1-c^2} \xi  & 0 & \eta  \\
 0 & 0 & 0 & 0 & \eta  & 0 \\
\end{array}
\right), } $$ $${\footnotesize K = \left(
\begin{array}{cccccc}
 0 & -i x & 0 & -i c y & 0 & 0 \\
 i x & 0 & -i y & 0 & 0 & 0 \\
 0 & i y & 0 & -i c x & 0 & 0 \\
 i c y & 0 & i c x & 0 & -i \sqrt{1-c^2} \xi  & 0 \\
 0 & 0 & 0 & i \sqrt{1-c^2} \xi  & 0 & -i \eta  \\
 0 & 0 & 0 & 0 & i \eta  & 0 \\
\end{array}
\right). } $$ We will now quantise this counterexample.  Define hermitian $X = \left( \begin{array}{cc}
0 & 1 \\
1 & 0
\end{array} \right)$ and $Y = \left( \begin{array}{cc}
0 & i \\
-i & 0
\end{array} \right),$ and consider $L(s) = s X \otimes H + Y \otimes K $, where $s \in \mathbb{R}, s > 1. $ We find that $L(s) = \left( \begin{array}{cc}
0 &M(s) \\
M(s)^{T} & 0
\end{array} \right), $ where $M(s) =$ $$ {\footnotesize  \left(
\begin{array}{cccccc}
 0 & (s+1) x & 0 & c (s+1) y & 0 & 0 \\
 (s-1) x & 0 & (s+1) y & 0 & 0 & 0 \\
 0 & (s-1) y & 0 & -c (s-1) x & 0 & 0 \\
 c (s-1) y & 0 & -c (s+1) x & 0 & \sqrt{1-c^2} \xi  (s+1) & 0 \\
 0 & 0 & 0 & \sqrt{1-c^2} \xi  (s-1) & 0 & \eta  (s+1) \\
 0 & 0 & 0 & 0 & \eta  (s-1) & 0 \\
\end{array}
\right).}$$ The eigenvalues of $L(s)$ will be plus/minus the singular values of $M(s)$, but moreover will be plus/minus the square root of the eigenvalues of $M(s)M(s)^{T}$. So by showing that values exist of $s$ for which $M(s)M(s)^{T}$ has simple eigenvalues, we will show the same for $L(s)$. Firstly, $\mbox{Det}(M(s)) = -c^2 \eta ^2 (s-1)^3 (s+1)^3 \left(x^2+y^2\right)^2$, which for the allowed values of $c, \eta,$ and $s$ is non-zero. 

Our approach is to evaluate the discriminant of the characteristic polynomial of $M(s)M(s)^{T}$ for $s = 2, 3 ,4$, and show that for any values of the parameters of the counterexample, at least one of these values of $s$ gives a non-zero discriminant. Therefore there is always some $sX$ and $Y$ for which $MM^{T}$ and thus $L$ have a non-zero discriminant, and therefore only simple eigenvalues. The process involves lengthy calculation of characteristic polynomials and their discriminants, done using Mathematica, and a Gr{\"o}bner basis calculation, done using Macaulay2. The code for these calculations may be be found at \texttt{github.com/blaw381/Quantum-Kippenhahn-Examples}.

\section{Quantisation of Laffey's counterexample}\label{sec_laffey}

Here is the counterexample given by Laffey in \cite{laffey}: \begin{equation*} 
\begin{split} H &=  {\small \left( \begin{array}{cccccccc}
-122 & 0 & 12 & 18 & -30 & 18 & 26 & 10 \\
0 & -122 & -6 & -12 & -16 & -28 & 20 & -16 \\
12 & -6  & -218 & 0 & 44 & 8 & 24 & 12 \\
18 & -12 & 0 & -218 & -2 & -34 & -10 & 22 \\
-30 & -16 & 44 & -2 & -216 & 0 & -12 & -8 \\
18 & -28 & 8 & -34 & 0 & -216 & -8 & 36 \\
26 & 20 & 24 & -10 & -12 & -8 & -120 & 0 \\
10 & -16 & 12 & 22 & -8 & 36 & 0 & -120 
\end{array} \right)} ,  \\  &\qquad \qquad \; K ={\small \left( \begin{array}{cccccccc}
-4 & 0 & 0 & 0 & 0 & 0 & 0 & 0 \\
0 & -4 & 0 & 0 & 0 & 0 & 0 & 0 \\
0 & 0 & 4 & 0 & 0 & 0 & 0 & 0 \\
0 & 0 & 0 & 4 & 0 & 0 & 0 & 0 \\
0 & 0 & 0 & 0 & -8 & 0 & 0 & 0 \\
0 & 0 & 0 & 0 & 0 & -8 & 0 & 0 \\
0 & 0 & 0 & 0 & 0 & 0 & 8 & 0 \\
0 & 0 & 0 & 0 & 0 & 0 & 0 & 8
\end{array} \right)}. \end{split} \end{equation*} Quantise with $X = \left( \begin{array}{cc}
1 & 0 \\
0 & 0
\end{array} \right)$ and $Y = \left( \begin{array}{cc}
0 & 1 \\
1 & 0
\end{array} \right) $ to get $L = X \otimes H + Y \otimes K. $

\begin{prop}
$L$ has only simple eigenvalues.
\end{prop}  

\begin{proof}

Observe that $L = \left( \begin{array}{cc}
H & K \\
K & 0
\end{array} \right) $, and notice that the lower blocks trivially commute. We can therefore use Lemma \ref{lemma_block_det} and conclude that $\mbox{Det}(L) = \mbox{Det}(-K^{2}) \neq 0.$ Therefore all of the eigenvalues of $L$ are non-zero, and we can construct the auxiliary matrix $$\widetilde{L}(\lambda) = H + \frac{1}{\lambda}K^{2}, $$ corresponding to some eigenvalue $\lambda$ of $L$. Note that for the purpose of this auxiliary matrix, $\lambda$ is a fixed constant. By Lemma \ref{lemma_aux}, we know that $\lambda$ is an eigenvalue of $\widetilde{L}(\lambda)$, and our aim is to determine the multiplicity of $\lambda$ with respect to $\widetilde{L}(\lambda)$ . Note Lax's theorem \cite{lax}, that a symmetric matrix has repeated eigenvalues if and only if it commutes with a non-zero skew-symmetric matrix. We can take the commutator of $\lambda \widetilde{L}(\lambda)$ with a skew symmetric matrix $S = \left\lbrace s_{ij} \right\rbrace$. We have multiplied by $\lambda$ to improve the clarity of the resulting commutator. This commutator has 36 distinct elements, a sample of which is listed here: 
 {\small \begin{align*}
 &(1) \qquad -4 \lambda  (6 s_{1,3}+9 s_{1,4}-15 s_{1,5}+9 s_{1,6}+13 s_{1,7}+5 s_{1,8}), \\
&(2) \qquad 2 (8 \lambda  s_{1,2}-22 \lambda  s_{1,3}+\lambda  s_{1,4}+47 \lambda  s_{1,5}+6 \lambda  s_{1,7}+4 \lambda  s_{1,8} \\ 
 &\qquad \qquad +6 \lambda  s_{3,5} +9 \lambda  s_{4,5}-9 \lambda  s_{5,6}-13 \lambda  s_{5,7}-5 \lambda  s_{5,8}-24 s_{1,5}), \\
&(3) \qquad 2 \lambda  (9 s_{1,2}+48 s_{2,4}+s_{2,5}+17 s_{2,6}+5 s_{2,7}-11 s_{2,8}-3 s_{3,4}+8 s_{4,5}+14 s_{4,6}-10 s_{4,7}+8 s_{4,8}), \\
& (4) \qquad 2 (9 \lambda  s_{1,2}-4 \lambda  s_{2,3}+17 \lambda  s_{2,4}+47 \lambda  s_{2,6}+4 \lambda  s_{2,7}-18 \lambda  s_{2,8} \\
 & \qquad  \qquad -3 \lambda  s_{3,6}-6 \lambda  s_{4,6}-8 \lambda  s_{5,6}-10 \lambda  s_{6,7}+8 \lambda  s_{6,8}-24 s_{2,6}), \\
& (5) \qquad 4 \lambda  (3 s_{1,2}+24 s_{2,3}-11 s_{2,5}-2 s_{2,6}-6 s_{2,7}-3 s_{2,8}+3 s_{3,4}+4 s_{3,5}+7 s_{3,6}-5 s_{3,7}+4 s_{3,8}), \\
& (6) \qquad 2 (5 \lambda  s_{1,4}+9 \lambda  s_{1,8}-8 \lambda  s_{2,4}-6 \lambda  s_{2,8}+6 \lambda  s_{3,4}+4 \lambda  s_{4,5} \\ 
&\qquad  \qquad -18 \lambda  s_{4,6}-49 \lambda  s_{4,8}-\lambda  s_{5,8}-17 \lambda  s_{6,8}-5 \lambda  s_{7,8}-24 s_{4,8}).
\end{align*} }
See \texttt{github.com/blaw381/Quantum-Kippenhahn-Examples} for the relevant code. We will suppose that each of these elements vanishes, giving us 36 equations, and solve for possible values of the $s_{i,j}$. Since $\lambda$ is a non-zero constant in the context of the auxiliary matrix, and each of the these polynomials is linear in the $s_{i,j}$, solving for the $s_{i,j}$ using linear algebra is immediately possible.  It is easy to confirm using any computer algebra system that this system of equations has full rank with respect to the $s_{i,j}$, and so we must conclude that $ s_{i,j} = 0$ for all $i,j$. Therefore $\lambda \widetilde{L}(\lambda)$, and hence also $\widetilde{L}(\lambda)$, must have only simple eigenvalues. In particular $\lambda$ must be a simple eigenvalue of $\widetilde{L}(\lambda)$ and hence by Lemma \ref{lemma_aux} also a simple eigenvalue of $L$. Since we have made no special assumptions about $\lambda$, we conclude that $L$ must have only simple eigenvalues.
\end{proof}

 \section{Final remarks on quantisation}
While Theorem \ref{thm_quant_kipp_intro} fixes the disproven Kippenhahn conjecture, the obtained quantities and thus the bounds on the size $n$ of $X$ and $Y$ needed are exponential in the size $d$ of $A$ and $B$. These bounds are likely far from optimal, a view suggested by the small size of the $X$ and $Y$ which we have found for the Kippenhahn conjecture counterexamples. There is a natural matrix convex set called a free spectrahedron associated to the matrices $A$ and $B$ in the counterexample family, consisting of a formal union of symmetric matrices $X$ and $Y$ of every size $n \in \mathbb{N}$ for which $L(X,Y)$ is positive semidefinite. The geometry of the boundary of this region is determined on the $n = d$ level; in fact, every boundary point can be naturally compressed to a $d \times d$ boundary point (cf also \cite[Section~3]{matricial_relaxation} or \cite{spectrahedral_containment}). We are thus led to consider that $n = d$ should suffice as a bound for the Quantum Kippenhahn Theorem. On the other hand, a recent result of Huber and Netzer \cite{polytopes} in a similar situation suggests that $n = 2$ may be sufficient in all cases. It would be interesting to know how much the existing bounds can be tightened, and if the desired property emerges at the first level of non-commutativity for $X$ and $Y$, that is for $n = 2$.  Alternatively, it would also be interesting to find matrices $A$ and $B$ which generate $M_{d}(\mathbb{C})$, such that for generic $2 \times 2 $ hermitian $X$ and $Y$, $\mbox{Det}(I + A \otimes X + B \otimes Y)$ is a square. Preliminary numerical investigation suggests that almost any $2 \times 2$ hermitian $X$ and $Y$ without common eiegenspaces seems to break the square determinant property, when applied to an $8 \times 8$ Kippenhahn counterexample of the type generated by Theorem \ref{thm_construct_simple}. We therefore conjecture that the smallest such $A$ and $B$ have to be of size at least 10. 

%%%%%%%% Bibliography %%%%%%%%%%%%%%%%%%%%%%%%%%%%%%%%%%%%%%%%%%%%%%%%%

\end{document}